\documentclass[11pt]{article}
\usepackage[margin=0.75in]{geometry}

\usepackage{amssymb,amsmath,mathtools,amsthm,thmtools,amsfonts,dsfont}
\usepackage{thm-restate}

\usepackage{url} 
\usepackage{hyperref}
\hypersetup{
colorlinks,
linkcolor={red}, 
citecolor={black},
urlcolor={blue!60!black},
pdftitle={Improved bounds on completion of partial Latin squares},
pdfauthor={Allsop, Bowtell, Lesgourgues, and Petrova}}

    \usepackage{listings}
    \usepackage{xcolor}
    \usepackage{tikz}

    \definecolor{pythonblue}{RGB}{0,0,180}
    \definecolor{pythongreen}{RGB}{0,120,0}
    \definecolor{pythonred}{RGB}{170,30,30}
    \definecolor{pythonpurple}{RGB}{120,0,120}
    \definecolor{pythongray}{RGB}{100,100,100}

    \lstdefinestyle{coloredpython}{
        language=Python,
        basicstyle=\ttfamily\footnotesize,
        keywordstyle=\color{pythonblue}\bfseries,
        commentstyle=\color{pythongreen}\itshape,
        stringstyle=\color{pythonred},
        morekeywords=[2]{
            print,float,len,range,
            array,asarray,where,linspace,argmin,sqrt,
            symbols,Rational,Float,diff,nsolve
        },
        keywordstyle=[2]\color{pythonpurple},
        numbers=left,
        numberstyle=\tiny\color{pythongray},
        numbersep=8pt,
        frame=single,
        breaklines=true,
        showstringspaces=false,
        tabsize=4,
        columns=fullflexible
    }

    \lstdefinestyle{python}{
        language=Python,
        basicstyle=\ttfamily\footnotesize,
        keywordstyle=\bfseries,
        commentstyle=\itshape\color{gray},
        stringstyle=\color{gray!70!black},
        numbers=left,
        numberstyle=\tiny\color{gray},
        numbersep=8pt,
        frame=single,
        breaklines=true,
        breakatwhitespace=false,
        showstringspaces=false,
        tabsize=4,
        columns=fullflexible
    }

    \lstdefinestyle{output}{
        basicstyle=\ttfamily\scriptsize,
        numbers=none,
        frame=single,
        breaklines=true,
        showstringspaces=false,
        columns=fullflexible,
        keepspaces=true,
        aboveskip=1em,
        belowskip=1em
    }    

    \definecolor{mathematicablue}{RGB}{0,0,180}
    \definecolor{mathematicagreen}{RGB}{0,120,0}
    \definecolor{mathematicared}{RGB}{170,30,30}
    \definecolor{mathematicapurple}{RGB}{120,0,120}
    \definecolor{mathematicagray}{RGB}{100,100,100}
    
    \lstdefinelanguage{Wolfram}{
        morekeywords={
            ClearAll,Piecewise,Sqrt,NMinimize,Minimize,
            WorkingPrecision,AccuracyGoal,PrecisionGoal,
            Print,First,N,Reals
        },
        sensitive=true,
        morecomment=[s]{(*}{*)},
        morestring=[b]"
    }
    
    \lstdefinestyle{wolfram}{
        language=Wolfram,
        basicstyle=\ttfamily\footnotesize,
        keywordstyle=\color{blue}\bfseries,
        commentstyle=\color{green!50!black}\itshape,
        stringstyle=\color{red!70!black},
        numbers=left,
        numberstyle=\tiny\color{gray},
        numbersep=8pt,
        frame=single,
        breaklines=true,
        showstringspaces=false,
        columns=fullflexible,
        keepspaces=true
    }

\usepackage[utf8]{inputenc}
\usepackage{color}
\usepackage{float} 
\usepackage{subcaption} 

\usepackage{enumitem}
\setlist[itemize]{topsep=.1in,itemsep=.1in}
\setlist[enumerate]{topsep=.1in,itemsep=.1in}
\setlist[enumerate,1]{label={(\roman*)}}
\setlist[enumerate,2]{label={(\alph*)}}
\setlist[enumerate,3]{label={(\arabic*)}}

\setlist[itemize]{nolistsep,noitemsep, topsep=0pt}

\newcommand{\eps}{\varepsilon}
\renewcommand{\epsilon}{\varepsilon}

\newcommand{\mc}[1]{\mathcal{#1}}%

\DeclareMathOperator*{\argmin}{argmin}

\DeclarePairedDelimiter{\parens}{(}{)}
\DeclarePairedDelimiter{\set}{\{}{\}}
\DeclarePairedDelimiter{\brackets}{[}{]}

\DeclarePairedDelimiter{\abs}{\lvert}{\rvert}   
\DeclarePairedDelimiter{\size}{\lvert}{\rvert}   

\usepackage[noabbrev,capitalise]{cleveref}

\theoremstyle{plain}
\newtheorem{thm}{Theorem}[section]
\newtheorem{lem}[thm]{Lemma}
\newtheorem{prop}[thm]{Proposition}

\newtheorem{cor}[thm]{Corollary}

\newtheorem{conj}[thm]{Conjecture}

\newenvironment{proofclaim}[1][\proofname]
{\proof[#1]}
{\endproof}

\declaretheorem[
  style=plain,
  name=Claim,
  within=theorem,
]{claim}

\newenvironment{nonlateproof}[1]
 {%
  \renewcommand{\theclaim}{\ref*{#1}.\arabic{claim}}%
  \begin{proof}%
 }
 {\end{proof}}

\usepackage{etoolbox}
 \AtEndEnvironment{proof}{\setcounter{claim}{0}}

\theoremstyle{plain} 
\newcommand{\thistheoremname}{}
\newtheorem{genericthm}{\thistheoremname}

\theoremstyle{definition}
\newtheorem{definition}[thm]{Definition}

\crefname{equation}{}{}
\crefname{lem}{Lemma}{Lemmas}
\crefname{claim}{Claim}{Claims}
\crefname{thm}{Theorem}{Theorems}
\crefname{prop}{Proposition}{Propositions}
\crefname{enumi}{}{}
\crefname{lstlisting}{Code}{Codes}

\usepackage[dvipsnames, table]{xcolor}

\newcommand{\Zout}{Z_{\textnormal{out}}}
\newcommand{\Zin}{Z_{\textnormal{in}}}
\newcommand{\Pout}{P_{\textnormal{out}}}
\newcommand{\Pin}{P_{\textnormal{in}}}
\newcommand{\Nout}{N_{\textnormal{out}}}
\newcommand{\Nin}{N_{\textnormal{in}}}

\newcommand{\zout}{z_{\textnormal{out}}}
\newcommand{\zin}{z_{\textnormal{in}}}
\newcommand{\pout}{p_{\textnormal{out}}}
\newcommand{\pin}{p_{\textnormal{in}}}
\newcommand{\nout}{n_{\textnormal{out}}}
\newcommand{\nin}{n_{\textnormal{in}}}

\newcommand{\dem}{\textnormal{dem}}
\newcommand{\disDP}{\textnormal{dis}_{\textnormal{DP}}}

\title{Improved bounds on completion of partial Latin squares}

\author{ 
Jack Allsop 
\thanks{Institut f\"{u}r Mathematik, Freie Universit\"{a}t Berlin, Germany. {\tt allsop@mi.fu-berlin.de}}
\and
Candida Bowtell
\thanks{School of Mathematics, University of Birmingham, Edgbaston, Birmingham, UK. {\tt c.bowtell@bham.ac.uk.}} 
\and
Thomas Lesgourgues
\thanks{School of Mathematics and Statistics, University of New South Wales, Australia. {\tt t.lesgourgues@unsw.edu.au}}
\and
Kalina Petrova
\thanks{Institute of Science and Technology Austria (ISTA), Klosterneuburg,
Austria. {\tt kalina.petrova@ist.ac.at}} }
\date{\today}

\usepackage{tikz}
\usepackage{calc}
\usetikzlibrary{calc,backgrounds,decorations.pathreplacing}
\usepackage{ifthen}

\tikzstyle{vertex}=[circle, draw, fill=black, inner sep=0pt, minimum size=4pt]
\newcommand{\vertex}{\node[vertex]}

\newcommand{\ColoredTriangle}[4]
{\begin{tikzpicture}[x=#1 cm, y=#1 cm]
    \vertex (a) at (-30:1) [label=180:]{};
    \vertex (b) at (90:1) [label=90:]{};
    \vertex (c) at (210:1) [label=0:]{};

    \path 
        (a) edge[#2] (c)
        (a) edge[#3] (b)
        (c) edge[#4] (b);

\end{tikzpicture}}

\newcommand{\ColoredEdge}[2]
{\begin{tikzpicture}[x=#1 cm, y=#1 cm]

    \path (0,0) edge[color=#2] (1,0);

\end{tikzpicture}}

\newcommand{\DischargingQuasiNeg}[1]
{\begin{tikzpicture}[x=#1 cm, y=#1 cm]
    \vertex (a) at (-30:1) [label=0:$v$]{};
    \vertex (b) at (90:1) [label=90:$u$]{};
    \vertex (c) at (210:1) [label=180:$w$]{};

    \path 
        (a) edge[color=blue] node[below]{$b$} (c)
        (a) edge[color=red] node[right]{$r_2$} (b)
        (c) edge[color=red] node[left]{$r_1$} (b);       

\end{tikzpicture}}

\newcommand{\DischargingQuasiPos}[1]
{\begin{tikzpicture}[x=#1 cm, y=#1 cm]
    \vertex (a) at (-30:1) [label=0:$v$]{};
    \vertex (b) at (90:1) [label=90:$u$]{};
    \vertex (c) at (210:1) [label=180:$w$]{};

    \path 
        (a) edge[color=blue] node[below]{$b$} (c)
        (a) edge[color=red] node[right]{$r$} (b)
        (c) edge[color=blue] node[left]{} (b);       

\end{tikzpicture}}

\newcommand{\DischargingQuasiPosBis}[1]
{\begin{tikzpicture}[x=#1 cm, y=#1 cm]
    \vertex (a) at (-30:1) [label=0:$v$]{};
    \vertex (b) at (90:1) [label=90:$u$]{};
    \vertex (c) at (210:1) [label=180:$w$]{};

    \path 
        (a) edge[color=blue] node[below]{$b$} (c)
        (a) edge[color=blue] node[right]{} (b)
        (c) edge[color=red] node[left]{$r$} (b);       

\end{tikzpicture}}

\newcommand{\RedEdges}[1]
{\begin{tikzpicture}[x=#1 cm, y=#1 cm]

    \vertex (x) at (0,0) [label=0:$u$]{};
    \vertex (y) at (-1,0) [label=200:$v$]{};
    \path (x) edge[color=red, line width=1.5] node[below]{$r$} (y);
    
    \vertex(lf1) at (-2,1)[]{};
    \path (y) edge[color=blue]  (lf1)
          (x) edge[color=red]  (lf1);
    \begin{scope}[on background layer]
        \draw[fill=gray!50, opacity=0.75] (-2,1) ellipse (0.35 and 0.1) node[above=6pt]{$W^-_u$};
    \end{scope}

    \vertex(rf1) at (1,1)[]{};
    \path (y) edge[color=red]  (rf1)
          (x) edge[color=blue]  (rf1);
    \begin{scope}[on background layer]
        \draw[fill=gray!50, opacity=0.75] (1,1) ellipse (0.35 and 0.1) node[above=6pt]{$W^-_v$};
    \end{scope}

    \vertex(sk1) at (-0.5,1)[]{};
    \path (y) edge[color=blue]  (sk1)
          (x) edge[color=blue]  (sk1);
    \begin{scope}[on background layer]
        \draw[fill=gray!50, opacity=0.75] (sk1) ellipse (0.35 and 0.1) node[above=6pt]{$W^+$};
    \end{scope}

\end{tikzpicture}}

\begin{document}

\maketitle
\begin{abstract}
A Latin square of order $n$ is an $n \times n$ array filled with $n$ symbols so that each symbol appears exactly once in every row and column. A partial Latin square of order $n$ is an $n
\times n$ array whose cells are either empty or filled in such a way that
each symbol appears at most once in every row and column, and at most $n$ distinct symbols are used. In 1983, Daykin and H{\"a}ggkvist conjectured that every partial Latin square in which each row and column contains at most $n/4$ symbols, and each symbol is used at most $n/4$ times, can be completed to a Latin square. We prove that every partial Latin square in which each row and column contains at most $0.231n$ symbols, and each symbol is used at most $0.231n$ times, can be completed to a Latin square, significantly improving the previous best-known bound of $0.08n$, obtained by Fu and Weng. This problem can be seen as a partite analogue of the Nash-Williams conjecture concerning triangle decompositions of dense graphs, recently proved in the breakthrough work of Delcourt and Postle. Our proof uses a `discharging' strategy, adapting the approach of Delcourt and Postle, combined with a novel method to achieve a `balancedness' property, required for the partite setting.
\end{abstract}


\section{Introduction}

A \emph{Latin square} of order $n$ is an $n \times n$ array with entries
from a set of $n$ symbols (often taken to be $[n]:=\{1,2,\dots,n\}$) with
the property that each symbol appears exactly once in every row and every
column. The formal mathematical study of Latin squares was initiated by Euler~\cite{euler} in the 18th Century, and has led to a vast body of work. In this manuscript we are interested in the problem of completing a partially filled $n \times n$ array to a Latin square of order $n$. We define a \emph{partial Latin square} of order $n$ to be an $n
\times n$ array whose cells are either empty or filled in such a way that
each symbol appears at most once in every row and every column, and at most $n$ distinct symbols are used. We say that a partial Latin square is {\em  completable} if we can add symbols to complete the array to a Latin square.

In 1981, Smetaniuk~\cite{S1981}, and independently Anderson and Hilton~\cite{AH1983}, showed that a partial Latin square with at most $n-1$ entries can always be completed to a Latin square. Without further restriction on where the symbols are placed, or which symbols are used, this bound is best possible, noting, for example, that a placement of symbols $1, \ldots, n-1$ in the first $n-1$ columns of the first row, and symbol $n$ into the final column of the second row cannot be completed. Moreover, this example shows that even when each symbol and each column is used only once, we cannot always complete a partial Latin square. To this end, a natural consideration is to limit the number of filled cells using any particular row, column or symbol.
In 1983, Daykin and H{\"a}ggkvist~\cite{DH1983} conjectured that any partial Latin square in which each row, column, and symbol is in at most $n/4$ entries is completable. Subsequent work by Chetwynd and H\"{a}ggkvsit~\cite{CH1985}, Gustavsson~\cite{G1991} and finally Bartlett~\cite{B2013} established a version of the conjecture with the constant $1/4$ replaced by an absolute constant $c\approx 10^{-4}$, showing that a partial Latin square with a quadratic number of entries is completable when these are sufficiently well spread over the rows, columns, and symbols.

To discuss more recent progress, it is convenient to switch viewpoints. A Latin square of order $n$ is equivalent to a $K_3$-decomposition of the complete $3$-partite graph $K_{n,n,n}$, that is, a partition of the edges of $K_{n,n,n}$ into copies of $K_3$, where the three parts of $K_{n,n,n}$ represent the rows, columns, and symbols of the Latin square. 
Similarly, any partial Latin square $P$ corresponds to a fixed subgraph $G \subseteq K_{n,n,n}$ which has a $K_3$-decomposition, and the completion problem for $P$ amounts to finding a $K_3$-decomposition in $K_{n,n,n} - E(G)$, the tripartite complement of $G$.  
Given a tripartite graph $G$ with vertex partition $(V_1,V_2,V_3)$ and a vertex $v\in V_i$, we denote \(d(v,V_j):= \lvert \set*{uv\in G\colon u\in V_j}\rvert, \) for each $i,j \in [3]$ and we define the {\em minimum partite-degree} of $G$ by \(\hat\delta(G) := \min\set*{d(v, V_j ) \colon j \in [3], v \in V (G)\setminus V_j}\).

The conjecture of Daykin and H\"{a}ggkvist is therefore implied by the following stronger one, first presented in this way by Barber, K\"{u}hn, Lo, Osthus and Taylor~\cite{BKLOT2017}, but implicitly present in the work of Gustavsson~\cite{G1991}. We say that a tripartite graph with vertex partition $(V_1,V_2,V_3)$ is $K_3$-divisible if for each $v \in V_i$, we have $d(v,V_{j}) = d(v,V_{h})$, whenever $\{i,j,h\} = \{1,2,3\}$. Note that $K_3$-divisibility is a necessary condition for $G$ to have a $K_3$-decomposition, since each copy of $K_3$ containing $v \in V_i$ contains precisely one edge joining $v$ to each of the other two parts $V_j$ and $V_h$. 
\begin{conj}\label{conj:decompo}
    Suppose $G$ is a $K_3$-divisible tripartite graph on vertex set $(V_1,V_2,V_3)$ with $|V_1|=|V_2|=|V_3|=n$, where $n$ is sufficiently large. 
    If $\hat{\delta}(G)\geq\frac{3n}{4}$, then $G$ contains a $K_3$-decomposition.
\end{conj}

This conjecture is stronger than that of Daykin and H\"aggkvist, as it does not require the complement of $G$ in $K_{n,n,n}$ to admit a $K_3$-decomposition.

\cref{conj:decompo} concerns the minimum partite-degree {\em threshold} for $K_3$-decompositions. Let $\hat{\delta}^*$ denote the infimum over all $c$ such that every $K_3$-divisible tripartite graph $G$ with vertex partition $(V_1,V_2,V_3)$ satisfying $|V_1|=|V_2|=|V_3|=n$ with $n$ sufficiently large and $\hat{\delta}(G)\geq cn$ admits a $K_3$-decomposition. Then \cref{conj:decompo} is equivalent to the assertion that $\hat{\delta}^* \leq \frac{3}{4}
$. As we detail further below, this bound is best possible by a construction of Wanless~\cite{W2002}.

An important and related notion is the {\em fractional relaxation} of the $K_3$-decomposition problem, in which each copy of $K_3$ may be assigned a fractional weight in $\mathbb{R}_{\geq 0}$ (rather than simply being selected or not), subject to the condition that the total weight of the copies containing any given edge is exactly one. Formally, with $\mc{T}(G)$ denoting the family of copies of $K_3$ (also referred to as triangles from now on) in $G$, a fractional $K_3$-decomposition of $G$ is an assignment of non-negative weights $\phi:\mc{T}(G)\to\mathbb{R}_{\geq 0}$ to the triangles in $G$ such that
\[\sum_{T\in\mc{T}(G)\colon e\in T}\phi(T)=1\text{ for all }e\in E(G).\]
We define the threshold \(\hat{\delta}^*_f\) as the infimum over all $c$ such that every $K_3$-divisible tripartite graph on vertex set $(V_1,V_2,V_3)$ with $|V_1|=|V_2|=|V_3|=n$ for $n$ sufficiently large and $\hat{\delta}(G) \geq cn$ contains a fractional $K_3$-decomposition. Whilst the fractional problem is interesting in its own right, much of its significance comes from classical absorption results that relate fractional decomposition thresholds to their integral counterparts. In the partite setting, such a result was obtained by Barber, K{\"u}hn, Lo, Osthus, and Taylor~\cite{BKLOT2017} in the more general setting of $r$-partite graphs, building on the seminal work of Haxell and R{\"o}dl~\cite{HR01}.

\begin{thm}[\cite{BKLOT2017}, Corollary 1.6]
Let $\varepsilon > 0$ and let $n$ be a sufficiently large integer. Suppose $G$ is a $K_3$-divisible spanning subgraph of $K_{n,n,n}$ with $\hat{\delta}(G) \geq (\hat{\delta}^*_f+\varepsilon)n$. Then $G$ admits a $K_3$-decomposition.
\end{thm}

Thus progress on \cref{conj:decompo}, up to an asymptotic resolution, can be achieved by improving the best bounds on \(\hat{\delta}^*_f\). In 2019 Bowditch and Dukes~\cite{BD2019} proved that $\hat{\delta}^*_f\leq 0.96$, using algebraic methodology. Very recently Yu and Feng~\cite{YF2026} showed that $\hat{\delta}^*_f\leq 0.92$, extending previous work of Delcourt and Postle~\cite{DP2021progress} on the non-partite case. Indeed, much more is known in the non-partite setting which corresponds to the famous Nash-Williams conjecture~\cite{Nash1970}. The conjecture states that every $K_3$-divisible graph on $n$ vertices for $n$ sufficiently large with minimum degree $3n/4$ admits a $K_3$-decomposition (in the non-partite setting, a graph $G$ is $K_3$-divisible if $|E(G)|$ is divisible by $3$, and every vertex has even degree, the natural necessary conditions under which such a graph $G$ could hope to have a $K_3$-decomposition). The Nash-Williams conjecture has garnered a lot of attention and was widely regarded as one of the most central open problems in extremal design theory, until it was very recently resolved in astounding work by Delcourt and Postle~\cite{DP26NW}. Settling the problem asymptotically relied on determining the fractional decomposition threshold, whilst proving the conjecture in full was achieved via a lengthy stability analysis and separate treatment of the extremal cases. However, the threshold for a fractional $K_3$-decomposition in the non-partite setting (denoted here $\delta^*_f$) does not imply anything about $\hat\delta^*_f$, the partite degree threshold. As observed by Bowditch and Dukes~\cite{BD2019} for the triangle case, and by Montgomery~\cite{M2019} for general fractional $K_r$-decompositions, determining an upper bound on $\hat\delta^*_f$ is at least as difficult as obtaining the same threshold bound for the non-partite case. Indeed, recall that the \emph{tensor product} of two graphs $G_1$ and $G_2$, denoted $G_1 \times G_2$, is the graph with $V(G_1\times G_2) = V(G_1) \times V(G_2)$ where two vertices $(v_1, v_2)$ and $(u_1, u_2)$ are adjacent if and only if $v_1u_1 \in E(G_1)$ and $v_2u_2 \in E(G_2)$. Notice that the graph $G\times K_3$ is a tripartite graph with parts each of size $v(G)$ and that $\hat{\delta}(G\times K_3) = \delta(G)$. Moreover, $G\times K_3$ is partite-$K_3$-divisible for any $G$ by construction. Finally, we note that $G\times K_3$ admits a fractional $K_3$-decomposition if and only if $G$ does; the backward direction follows by using the same weight for a copy of $K_3$ in $G\times K_3$ as the weight of its original copy in $G$, while the forward direction follows by using as the weight of a $K_3$ in $G$ the average of all its partite copies in $G\times K_3$. In particular, we know that $\hat\delta^*_f\geq \delta^*_{f}\geq \frac34$; this lower bound comes from an editorial remark by Graham accompanying the original article of Nash-Williams~\cite{Nash1970}. Although the remark concerned $K_3$-decompositions, a simple folklore counting argument shows that the construction provided admits no fractional $K_3$-decomposition.

Nevertheless, in this manuscript, we make a significant improvement on the upper bound for $\hat\delta_f^*$, adapting ideas coming from the work of Delcourt and Postle~\cite{DP26NW}. Our main result bridges the gap between the upper and lower bounds considerably, reducing the previous upper bound of $0.92$ to under $0.769$.

\begin{thm}\label{thm:main}
There exists $\eta>0$ such that the following holds for every positive integer $n$. Let $G$ be a $K_3$-divisible spanning subgraph of $K_{n,n,n}$ with $\hat{\delta}(G) \geq (0.769-\eta)n$. Then $G$ admits a fractional $K_3$-decomposition. That is,
    \(\hat\delta^*_f< 0.769.\)
\end{thm}

Returning back to the setting of Latin squares, we immediately obtain the following corollary.
\begin{cor}\label{cor:completionLS}
    Every partial Latin square of sufficiently large order $n$ in which every row, column, and symbol occurs at most $0.231n$ times can be completed to a Latin square.
\end{cor}

Barber, K\"{u}hn, Lo, Osthus and Taylor~\cite{BKLOT2017} studied $\hat\delta^\eta$, the threshold for approximate decomposition: an $\eta$-approximate $K_3$-decomposition of $G$ is a set of edge-disjoint copies of $K_3$ covering all but at most $\eta n^2$ edges of $G$. For each $\eta>0$, let $\hat\delta^\eta$ denote the infimum over all $c$ such that every $K_3$-divisible tripartite graph $G$ with vertex partition $(V_1,V_2,V_3)$ satisfying $|V_1|=|V_2|=|V_3|=n$ with $n$ sufficiently large and $\hat{\delta}(G) \geq cn$ admits an $\eta$-approximate $K_3$-decomposition. A seminal result by Haxell and R\"odl~\cite{HR01} implies that $\hat\delta^*_f\geq\hat\delta^\eta$ for every $\eta>0$. Barber, K\"{u}hn, Lo, Osthus and Taylor then showed that $\hat\delta^\eta\geq\frac34$ for every $\eta>0$, thus providing an alternative proof of the known lower bound $\hat\delta^*_f\geq \frac34$. Moreover, Wanless~\cite{W2002} gave an explicit construction of a partial Latin square which is not completable, where every row, column, and symbol appears at most $n/4+1$ times, witnessing $\hat\delta^*\geq\frac34$. In~\cref{sec:LowerBound}, we provide a similar construction with a proof adapted for fractional $K_3$-decompositions. Define a {\em $k$-plex} to be a partial Latin square of order $n$ containing $kn$ entries such that every row, column, and symbol is used exactly $k$ times.

\begin{prop}\label{prop:LowerBound}
    For every $n\geq 8$ divisible by $4$, there exists a partial Latin square of order $n$, obtained from an $n/4$-plex by adding one single entry, such that its complement in $K_{n,n,n}$ admits no fractional $K_3$-decomposition.
\end{prop}

\paragraph{Additional note:} Independently, around the same time, He and Cheng~\cite{HeCheng} announced an upper bound of $0.8$ for $\hat\delta^*_f$. Their proof is notably different from ours, although they also use the dual linear programme obtained from Farkas' lemma. They also note that their methodology breaks at~$0.8$.

\paragraph{Organisation:}
In~\cref{sec:overview} we give an overview of our proof strategy, and discuss its connection to Delcourt and Postle's work~\cite{DP26NW} on the Nash-Williams conjecture, whilst highlighting the obstacles to using their methodology, and how we overcome them. The starting point of our approach is similar to that in~\cite{DP26NW}: we apply Farkas' lemma to pass to the dual setting, reducing the problem to showing that any assignment of weights to the edges for which every triangle has positive total weight must itself have positive total weight. \cref{s:balanced} is dedicated to our main structural ingredient: we show that such edge-weightings can be modified to have some additional desirable properties, most notably that each vertex has equal numbers of negative edges to the other two parts. Moreover, this modification can be obtained with only a small change of the total weight over all edges.
In~\cref{sec:main}, we introduce the non-uniform proportional discharging method, and use it to prove our main result,~\cref{thm:main}, by redistributing the weights among the edges in our edge-weighting so that each edge becomes positive. \cref{sec:LowerBound} is dedicated to the proof of~\cref{prop:LowerBound}. In~\cref{sec:FutureWork}, we discuss potential improvement routes for~\cref{thm:main} and related future research directions.

\section{Proof Overview}\label{sec:overview}
	
	Throughout this section, $G$ will be a $K_3$-divisible tripartite graph with vertex partition $(V_1, V_2, V_3)$ where $|V_1|=|V_2|=|V_3|=n$. We will take the indices of these sets modulo $3$, so that $V_4=V_1$ and $V_0=V_3$.

    The fractional $K_3$-decomposition problem can be stated as a linear programme, which makes it amenable to Farkas' lemma~\cite{MR1580578}. This allows us to work in the dual setting instead, where weights are assigned to the edges instead of the triangles. To introduce our main tool, we first need some definitions on edge-weight functions.
	
	We say that an edge-weight function (which we also refer to as an edge-weighting) $\phi:E(G)\to\mathbb{R}$ is:
	\begin{itemize}
		\item {\em totally-negative} if the total weight over all edges of $G$ is strictly negative, that is, \[\sum_{e\in E(G)}\phi(e)< 0.\]
		\item {\em triangle-covered} if every triangle has non-negative total weight, that is,\[\sum_{e\in T}\phi(e)\geq 0,\text{ for all }T\in \mc{T}(G).\]
		\item {\em balanced} if every vertex has the same number of negative edges into each part, that is, for every $i\in[3]$ and $v\in V_i$, we have
		\[\size*{\set*{u\in V_{i-1}\colon \phi(uv)<0}}=\size*{\set*{u\in V_{i+1}\colon \phi(uv)<0}}.\]
		\item {\em untied} if no two edges have the same weight, that is, 
		\[\phi(e)\neq \phi(f)\text{ for all }e\neq f\in E(G).\]
		\item {\em non-zero} if $\phi(e)\neq 0$ for all $e\in E(G)$.
	\end{itemize}
	
	We are now ready to state a version of Farkas' lemma adapted to our setting, due to Delcourt and Postle~\cite[Lemma 2.2]{DP26NW}.
	
	\begin{lem}[Farkas’ lemma applied to $K_3$-Decompositions]\label{lem:NoFrac} A graph $G$ does not admit a fractional $K_3$-decomposition if and only if there exists a triangle-covered and totally-negative edge-weighting of $G$. 
	\end{lem}
	
	Suppose that $G$ satisfies the minimum degree condition in \cref{thm:main} and also suppose, for a contradiction, that $G$ does not admit a fractional $K_3$-decomposition. By \cref{lem:NoFrac}, there exists a triangle-covered and totally-negative edge-weighting of $G$. The first step in our proof is to prove that there exists such an edge-weighting of $G$ that in addition is balanced, untied, and non-zero. This is the main result of \cref{s:balanced}. Given such an edge-weighting $\phi$, we then use the method of discharging to obtain a contradiction. Explicitly, we design a set of rules called discharging rules that creates a new edge-weighting from $\phi$ by moving the edge weights around, but not changing the overall weight. We then prove that after applying the discharging, each edge has positive weight, contradicting the totally-negative property. The application of the discharging method is done in \cref{sec:main}. The balanced property of $\phi$ is crucial for our discharging rule.

    The approach detailed above is heavily inspired by the recent resolution of the Nash-Williams conjecture by Delcourt and Postle~\cite{DP26NW}. In their solution of the fractional Nash-Williams conjecture, they began with an edge-weighting from \cref{lem:NoFrac} and then used this to prove that there exists a non-zero, untied, triangle-covered, totally-negative edge-weighting before using a discharging method to reach a contradiction. 
    The main new structural ingredient in our argument is the additional balancedness property, which is specific to our partite setting and was introduced to overcome the difficulties arising from the partite structure of the problem.
    Our discharging method also originates in the first-order proportional discharging developed by Delcourt and Postle in their work on the Nash-Williams conjecture~\cite{DP26NW}, and explained to us in more detail through personal communication. In its simplest form, this rule splits the weight of a positive edge equally between its two endpoints and distributes each half uniformly among the incident negative edges. We retain this proportional framework, but allow the split between the two endpoints to depend on the local configuration. In fact, the original uniform rule, combined with our balancedness property, already yields a substantial improvement on the previously known partite threshold; the non-uniform split is used to obtain the stronger bound proved here.
    Delcourt and Postle ultimately use a substantially more involved discharging rule to resolve the Nash-Williams problem.  
    We explored adapting it to the partite setting, but several new difficulties arise, and in particular we found partite configurations for which their rule does not seem to handle the partite structure well enough, we refer the reader to~\cref{sec:FutureWork} for additional details. Instead, our approach retains the proportional framework and exploits the additional balancedness of the edge-weighting. Thus this balancedness is the main new structural ingredient in our proof: it is obtained by translating the edge-weighting to minimise a given norm, and then using a max-flow/min-cut argument.

\section{Balanced weights}\label{s:balanced}

Throughout this section, let $G$ be a $K_3$-divisible tripartite graph with vertex partition $(V_1, V_2, V_3)$ where $|V_1|=|V_2|=|V_3|=n$ and where the indices of these sets are taken modulo $3$.
	The goal of this section is to prove \cref{prop:balanced} below, which states that for every triangle-covered edge-weighting $\phi$ of a $K_3$-divisible tripartite graph $G$, there exists a triangle-covered, balanced, untied, non-zero edge-weighting $\psi$ that is arbitrarily close to $\phi$ in terms of their sum over all the edges of $G$. For simplicity, we will refer to an edge as blue or red, respectively, depending on whether its weight is positive or negative.	
	The first step is to prove that there is a triangle-covered edge-weighting $\varphi$ that is arbitrarily close to $\phi$ and has the property that we can label each edge with weight $0$ as red or blue so that $\varphi$ is ``red-balanced'', meaning that for each $i \in \{1, 2, 3\}$ and each $v \in V_i$, the number of red edges incident to $v$ with an endpoint in $V_{i-1}$ is equal to the number of red edges incident to $v$ with an endpoint in $V_{i+1}$. This is done in \cref{lem:redbalanced}.
	We then apply a small perturbation to $\varphi$ to yield the desired edge-weighting $\psi$. We note that for our purposes, the edge-weighting $\varphi$ would have sufficed. However, the extra conditions of being untied and non-zero were critical for the proof of the Nash-Williams conjecture and may be useful for improving our results.

	\begin{prop}\label{prop:balanced}
		Let $\phi:E(G)\to\mathbb{R}$ be a triangle-covered edge-weighting of $G$. For every $\varepsilon>0$, there exists a triangle-covered, balanced, untied, and nonzero edge-weighting $\psi:E(G)\to\mathbb{R}$ such that $\abs*{\sum_{e\in E(G)}\phi(e)-\sum_{e\in E(G)}\psi(e)}\leq\varepsilon$. 
	\end{prop}
	
	To obtain the edge-weighting $\varphi$ described above, the idea is to {\em translate} $\phi$ according to the following definition.
	
	\begin{definition}
		Let $\phi:E(G)\to\mathbb{R}$ be an edge-weighting of $G$ and let $p:V(G)\to\mathbb{R}$. We denote by $\phi_p$ the {\em translation} of $\phi$ by $p$, defined by
		\[\phi_p(uv) := \phi(uv)-p(u)+p(v),\]
		for every $uv\in G$ such that $u\in V_i,v\in V_{i+1}$ for some $i\in[3]$.
	\end{definition}
	
	We first prove that any translation of a triangle-covered edge-weighting remains triangle-covered whilst maintaining the sum of the edge weights.
	
	\begin{lem}\label{lem:translate}
		If $\phi$ is a triangle-covered edge-weighting of $G$, then for every $p:V(G)\to\mathbb{R}$ the edge-weighting $\phi_p$ is also triangle-covered with $\sum_{e\in E(G)}\phi(e)=\sum_{e\in E(G)}\phi_p(e)$.
	\end{lem}
	
	\begin{proof}
		Fix $p:V(G)\to\mathbb{R}$. Observe that 
		\[\sum_{e\in E(G)}\phi_p(e) = \sum_{e\in E(G)}\phi(e) + \sum_{i=1}^3\sum_{v\in V_i} p(v)\cdot\brackets*{d(v,V_{i-1})-d(v,V_{i+1})}.\]
		As $G$ is $K_3$-divisible, we obtain that
		\[\sum_{e\in E(G)}\phi_p(e) = \sum_{e\in E(G)}\phi(e).\]
		Now, for a given triangle $T=v_1v_2v_3\in\mc{T}(G)$, with $v_i\in V_i$ for each $i\in[3]$, 
		\begin{align*}
			\sum_{e\in T}\phi_p(e)  
			=& \big[\phi(v_1v_2)-p(v_1)+p(v_2)\big]
			+ \big[\phi(v_2v_3)-p(v_2)+p(v_3)\big]
			+ \big[\phi(v_3v_1)-p(v_3)+p(v_1)\big]\\
			=&\sum_{e\in T}\phi(e)\geq 0,
		\end{align*}
		as desired.
	\end{proof}
	
	We next prove that for some choice of the function $p$ and some red/blue colouring of the edges of weight $0$ in $E(G)$, the edge-weighting $\phi_p$ is triangle-covered and the red-balanced property is satisfied.
	
	\begin{lem}\label{lem:redbalanced}
		If $\phi$ is a triangle-covered edge-weighting of $G$, then there exists $p:V(G)\to\mathbb{R}$ and a red/blue colouring $R\cup B$ of $E(G)$ such that $R$ is triangle-free, 
		\begin{equation}\label{e:cond1}
		\{e\colon \phi_p(e)<0\}\subseteq R\subseteq \{e\colon \phi_p(e)\leq 0\},
		\end{equation} 
		and the ``red-balanced'' property is satisfied, meaning that for each $i\in[3]$ and $v\in V_i$,
		\[\size*{\set{u\in V_{i-1}\colon uv\in R}}=\size*{\set{u\in V_{i+1}\colon uv\in R}}.\]
	\end{lem}
	
    \begin{nonlateproof}{lem:redbalanced}
		For a triangle-covered edge-weighting $\phi$ of $G$, define $\Phi(\phi)$ to be the family of all $\phi_p$ with $p : V(G) \to \mathbb{R}$. Also define $\mathbb V = \Phi(\mathbf{0})$ where $\mathbf{0} : E(G) \to \mathbb{R}$ denotes the $0$ function and note that $\mathbb V$ is a subspace of $\mathbb{R}^{E(G)}$.
		Now fix a triangle-covered edge-weighting $\phi$ of $G$. Since $\Phi(\phi)$ is a coset of the subspace $\mathbb V$, it is closed. Define $L : \Phi(\phi) \to \mathbb{R}$ by
		\[
		L(\phi_p) = \sum_{e \in E(G)} |\phi_p(e)|.
		\]        
        Fix some edge-weighting $\psi_0\in\Phi(\phi)$, and let 
        $\mathcal{K}:=\set*{\psi\in\Phi(\phi)\colon L(\psi)\leq L(\psi_0)}$.
        Observe that the set $\mathcal{K}$ is non-empty, closed, and bounded since for every $e\in G$ and every $\psi\in \mathcal{K}$ we have
        $|\psi(e)|\leq L(\psi)\leq L(\psi_0)$, and therefore $\mathcal{K}\subseteq [-L(\psi_0),L(\psi_0)]^{E(G)}$. It follows that $\mathcal{K}$ is compact.
        As $L$ is a continuous function, it attains its minimum on $\mathcal{K}$, which is also its minimum on $\Phi(\phi)$ by definition of $\mathcal{K}$. Thus we may set 
		\[\hat\phi:=\argmin\{L(\phi_p)\colon \phi_p\in \Phi(\phi)\}.\]
		
		We now have our candidate choice for the function $p$ in the lemma statement and so it remains to determine the red/blue colouring of $E(G)$. Define
		\begin{align*}
			&R_1 = \{e \in E(G) : \hat\phi(e) < 0\} \text{ and}\\
			&B_1 = \{e \in E(G) : \hat\phi(e) > 0\}.
		\end{align*}
		Our final red/blue colouring will satisfy $R_1 \subseteq R$ and $B_1 \subseteq B$, so it just remains to determine the colour of the edges that have zero weight in $\hat\phi$. We must set the colour of these edges so that we achieve the red-balanced property. 
		
		Define a relation $\prec$ on $V(G)$ where $u \prec v$ if $u \in V_i$ and $v \in V_{i+1}$ for some $i \in \{1, 2, 3\}$. Let $S\subseteq V(G)$ and define
		\begin{align*}
			&\Zout(S) = \{uv \in E(G) : u \prec v, u \in S, v \notin S, \hat\phi(uv) = 0\} \text{ and } \zout(S) = |\Zout(S)|,\\
			&\Zin(S) = \{uv \in E(G) : u \prec v, u \notin S, v \in S, \hat\phi(uv) = 0\} \text{ and } \zin(S) = |\Zin(S)|,\\
			&\Pout(S) = \{uv \in E(G) : u \prec v, u \in S, v \notin S, \hat\phi(uv) > 0\} \text{ and } \pout(S) = |\Pout(S)|,\\
			&\Pin(S) = \{uv \in E(G) : u \prec v, u \notin S, v \in S, \hat\phi(uv) > 0\} \text{ and } \pin(S) = |\Pin(S)|,\\
			&\Nout(S) = \{uv \in E(G) : u \prec v, u \in S, v \notin S, \hat\phi(uv) < 0\} \text{ and } \nout(S) = |\Nout(S)|,\\
			&\Nin(S) = \{uv \in E(G) : u \prec v, u \notin S, v \in S, \hat\phi(uv) < 0\} \text{ and } \nin(S) = |\Nin(S)|.\\
		\end{align*}
		So every edge of $G$ that has exactly one endpoint in $S$ lies in one of the above defined sets. In the case where $S = \{v\}$ is a singleton, then we will use $\Zout(v)$ instead of $\Zout(\{v\})$ and similarly for the other functions defined above. 
        We now fix $\varepsilon \in \mathbb{R}$ so that $|\varepsilon| < |\hat\phi(e)|$ for every $e\in E(G)$ with $\hat\phi(e)\neq0$, and define $p : V(G) \to \mathbb{R}$ by 
		\[
		p(v) = \begin{cases}
			\epsilon & \text{if } v \in S,\\
			0 & \text{otherwise}.
		\end{cases}
		\]
        Then 

        \begin{equation}\label{e:phi_change}
        |\hat\phi_p(uv)| = \begin{cases}
            |\hat\phi(uv)|+\eps & \text{if } \eps>0 \text{ and } uv \in \Zout(S)\cup\Zin(S)\cup\Pin(S)\cup\Nout(S),\\  
            |\hat\phi(uv)|-\eps & \text{if } \eps>0 \text{ and } uv \in \Pout(S)\cup\Nin(S),\\[.5em]       
            |\hat\phi(uv)|-\eps & \text{if } \eps<0 \text{ and } uv \in \Zout(S)\cup\Zin(S)\cup\Pout(S)\cup\Nin(S), \\
            |\hat\phi(uv)|+\eps & \text{if } \eps<0 \text{ and } uv \in \Pin(S)\cup\Nout(S).\\
        \end{cases}
        \end{equation}

        Here we have crucially used the fact that $|\eps|<|\hat\phi(e)|$ for every $e\in E(G)$ with $\hat\phi(e)\neq0$. Suppose that $\eps>0$. Then

        \newlength{\myphantomwidth}
        \settowidth{\myphantomwidth}{$\displaystyle
        =\sum_{\substack{uv \in E(G)\\\{u, v\} \cap S = \emptyset}}
         |\hat\phi_p(uv)|
        +\sum_{\substack{uv \in E(G)\\\{u, v\} \subseteq S}}
         |\hat\phi_p(uv)|
        $}

        \begin{align*}
    		L(\hat\phi_p) 
            &=\sum_{\substack{uv \in E(G)\\\{u, v\} \cap S = \emptyset}} |\hat\phi_p(uv)| 
            + \sum_{\substack{uv \in E(G)\\\{u, v\} \subseteq S}} |\hat\phi_p(uv)| 
            + \sum_{e \in \Zout(S)} |\hat\phi_p(e)|
            + \sum_{e \in \Nout(S)} |\hat\phi_p(e)|
            + \sum_{e \in \Pout(S)} |\hat\phi_p(e)|\\
            &\hspace{\myphantomwidth}
            + \sum_{e \in \Zin(S)} |\hat\phi_p(e)|
            + \sum_{e \in \Nin(S)} |\hat\phi_p(e)|
            + \sum_{e \in \Pin(S)} |\hat\phi_p(e)|\\
    		&=\sum_{\substack{uv \in E(G)\\\{u, v\} \cap S = \emptyset}} |\hat\phi(uv)|
            + \sum_{\substack{uv \in E(G)\\\{u, v\} \subseteq S}} |\hat\phi(uv)| 
            + \sum_{e \in \Zout(S)} \parens*{|\hat\phi(e)| + \epsilon}
            + \sum_{e \in \Nout(S)} \parens*{|\hat\phi(e)| + \epsilon}
            + \sum_{e \in \Pout(S)} \parens*{|\hat\phi(e)| - \epsilon}\\   
            &\hspace{\myphantomwidth}
            + \sum_{e \in \Zin(S)} \parens*{|\hat\phi(e)| + \epsilon}
            + \sum_{e \in \Nin(S)} \parens*{|\hat\phi(e)| - \epsilon}
            + \sum_{e \in \Pin(S)} \parens*{|\hat\phi(e)| + \epsilon}\\
    		&= L(\hat\phi) + \epsilon(\zout(S)+\zin(S)-\pout(S)+\pin(S)+\nout(S)-\nin(S)).
		\end{align*}

		By minimality of $\hat\phi$, we have $L(\hat\phi_p)\geq L(\hat\phi)$, hence 
		\begin{equation}\label{e:min_ineq}
		\zout(S)+\zin(S)-\pout(S)+\pin(S)+\nout(S)-\nin(S) \geq 0.
		\end{equation}
		Also, the $K_3$-divisibility of $G$ implies that
		\begin{equation}\label{e:divis_eq}
		\zout(S)+\pout(S)+\nout(S)=\zin(S)+\pin(S)+\nin(S).
		\end{equation}
		Combining \cref{e:min_ineq} and \cref{e:divis_eq}, we obtain that
		\begin{equation}\label{e:zout_prelim}
		\zout(S) \geq \nin(S)-\nout(S).
		\end{equation}
        Now suppose that $\eps<0$. By using \cref{e:phi_change} and running the above argument, we obtain
    	\begin{equation}\label{e:zin_prelim}
    		\zin(S) \geq \nout(S)-\nin(S).
    	\end{equation}
		
		For $v \in V(G)$, let $w(v) = n_{\textnormal{in}}(v)-n_{\textnormal{out}}(v)$ and observe that
		\begin{align*}
		w(S) :&= \sum_{v \in S} w(v) \\
		&= \sum_{v \in S} ((|\Nin(v) \cap S|+|\Nin(v) \cap (V(G) \setminus S)|)-(|\Nout(v) \cap S|+|\Nout(v) \cap (V(G) \setminus S)|)) \\
		&= \sum_{v \in S} |\Nin(v) \cap S| - \sum_{v \in S} |\Nout(v) \cap S| + \sum_{v \in S} |\Nin(v) \cap (V(G) \setminus S)| - \sum_{v \in S}|\Nout(v) \cap (V(G) \setminus S)| \\
		&= \nin(S)-\nout(S).
		\end{align*}
	Combining this with \cref{e:zout_prelim}, we obtain
	\begin{equation}\label{e:zout}
		\zout(S) \geq w(S).
	\end{equation}

	\definecolor{quantumblue}{RGB}{0,90,255}
	\definecolor{quantumviolet}{RGB}{135,40,255}
	\definecolor{quantumred}{RGB}{255,25,70}
	
	\newcommand{\quantumlogo}{%
		\tikz[baseline=(Q.base)]{%
			
			\foreach \dx/\dy in {
				-1.2/0, 1.2/0, 0/-1.2, 0/1.2,
				-.9/-.9, -.9/.9, .9/-.9, .9/.9,
				-1.8/0, 1.8/0, 0/-1.8, 0/1.8
			}{
				\node[
				inner sep=0pt,
				opacity=.10,
				text=quantumviolet,
				xshift=\dx pt,
				yshift=\dy pt
				] at (0,0)
				{\sffamily\bfseries\Large QUANTUM};
			}
			
			\node[inner sep=0pt] (Q) at (0,0) {%
				\sffamily\bfseries\Large
				\textcolor{quantumblue}{Q}%
				\textcolor{blue!65!quantumviolet}{U}%
				\textcolor{blue!30!quantumviolet}{A}%
				\textcolor{quantumviolet}{N}%
				\textcolor{quantumviolet!65!quantumred}{T}%
				\textcolor{quantumviolet!30!quantumred}{U}%
				\textcolor{quantumred}{M}%
			};
			
			\node[
			text=quantumblue,
			font=\scriptsize
			] at ([xshift=-2pt,yshift=4pt]Q.north west) {$\star$};
			
			\node[
			text=quantumviolet,
			font=\tiny
			] at ([xshift=5pt,yshift=5pt]Q.north) {$\star$};
			
			\node[
			text=quantumred,
			font=\scriptsize
			] at ([xshift=2pt,yshift=4pt]Q.north east) {$\star$};
			
			\node[
			text=quantumviolet,
			font=\tiny
			] at ([xshift=4pt,yshift=-3pt]Q.south east) {$\star$};
		}%
	}

	Recall that our goal is to determine the colours of the edges of zero weight to achieve the red-balanced property. Let $v \in V(G)$. If $\nin(v) > \nout(v)$, we need to colour red at least $\nin(v)-\nout(v)$ edges $uv$ of zero weight with $v \prec u$. Similarly, if $\nout(v) > \nin(v)$, then we need to colour red at least $\nout(v)-\nin(v)$ edges $uv$ of zero weight with $u \prec v$. By \cref{e:zout_prelim} and \cref{e:zin_prelim}, we know that we can achieve this locally at each vertex. However, a greedy approach may lead to a single zero edge being coloured both red and blue simultaneously, creating a ``quantum" 
    edge. To overcome this issue, we use network flows.
	
	Let $Z = \{e \in E(G) : \hat\phi(e) = 0\}$. Let $D$ be a digraph with vertex set $V(G) \cup \{s, t\}$, where $s$ and $t$ will be the source and sink vertices, respectively. For each edge $uv \in Z$ with $u \prec v$, we add the arc $(u, v)$ to $D$ with capacity $1$. For each $v \in V(G)$ with $w(v) > 0$, we add the arc $(s, v)$ to $D$ with capacity $w(v)$. For each $v \in V(G)$ with $w(v) < 0$, we add the arc $(v, t)$ to $D$ with capacity $-w(v)$. 
	
		\begin{claim}\label{cl:flow}
			There exists an integral $s$-$t$ flow $f$ in $D$ with capacity
			\[C:=\sum_{v\in V(G)}\max(w(v),0)=\sum_{v\in V(G)}\max(-w(v),0).\]
			Both $s$ and $t$ are saturated by $f$.
		\end{claim}
		\begin{proofclaim}
			Let $S' = S \cup \{s\}$ be an $s$-$t$ cut of $D$. The capacity $C(S')$ satisfies: 
			\begin{align*}
				C(S')&=\sum_{v\not\in S' \cup \{t\}}\max(w(v),0) + \sum_{v\in S}\max(-w(v),0) + \zout(S)\\
				&= \sum_{v\in V(G)}\max(w(v),0)  - \sum_{v\in S}\max(w(v),0) + \sum_{v\in S}\max(-w(v),0) + \zout(S)\\
				&= \sum_{v\in V(G)}\max(w(v),0)  - \sum_{v\in S}w(v) + \zout(S)\\
				&= C  - w(S) + \zout(S)\\
				&\geq C,
			\end{align*}
			by~\cref{e:zout}. By the integral max-flow min-cut theorem, there exists an integral flow from $s$ to $t$ with capacity at least $C$. Since the sum of the capacities of the arcs incident to $s$ is equal to $C$, any flow has capacity at most $C$. The claim follows.
		\end{proofclaim}
	
		Let $f$ be a flow of $D$ as in \cref{cl:flow} and let $R_0$ be the set of edges in $Z$ whose corresponding arcs have flow value $1$ in $f$. Set $\hat R = R_0 \cup R_1$ and $\hat B = E(G) \setminus \hat R$. By construction, 
		\[
		R_1 \subseteq \hat R \subseteq R_1 \cup Z,
		\]
		so the colouring $\hat R \cup \hat B$ satisfies \cref{e:cond1}. We next verify that the red-balanced property is satisfied. Let $v \in V(G)$ and first suppose that $w(v) > 0$. Let $f_v$ be the total flow into vertex $v$. Since all edges incident to $s$ are saturated by $f$ and the arc $(s, v)$ of $D$ has capacity $w(v)$, it follows that $f_v \geq w(v)$ and the number of edges $uv \in R_0$ with $u \prec v$ is $f_v-w(v)$. Since $(v, t)$ is not an arc of $D$, the number of edges $uv \in R_0$ with $v \prec u$ is $f_v$. Therefore,
        \begin{align*}
        &|\{uv \in R : u \prec v\}| = \nin(v)+|\{uv \in R_0 : u \prec v\}| = \nin(v)+f_v-w(v), \text{ and }\\
        &|\{uv \in R : v \prec u\}| = \nout(v)+|\{uv \in R_0 : v \prec u\}| = \nout(v)+f_v = \nin(v)+f_v-w(v),
        \end{align*}
        meaning that the red-balanced property is satisfied at $v$. By a similar argument, the red-balanced property is satisfied at $v$ if $w(v)<0$.

		Let $R \cup B$ be a colouring of $E(G)$ so that \cref{e:cond1} and the red-balanced property are satisfied with the smallest number of red edges. Note that such a colouring exists, since $\hat R \cup \hat B$ satisfies \cref{e:cond1} and the red-balanced property. We claim that $R$ is also triangle-free. If not, then there is a red triangle, say with edges $e_1, e_2, e_3$. Since by \cref{lem:translate} $\hat\phi$ is triangle-covered, it follows that $\hat\phi(e_i)=0$ for $i \in \{1, 2, 3\}$. Thus we can define $R' = R \setminus \{e_1, e_2, e_3\}$ and $B' = B \cup \{e_1, e_2, e_3\}$ and it is easy to check that this colouring satisfies \cref{e:cond1} and the red-balanced property with a smaller number of red edges than $R \cup B$, a contradiction. Thus, $R \cup B$ is also triangle-free, proving the lemma.

    \end{nonlateproof}
	
	We are now able to prove the main result for this section.
	
	\begin{proof}[Proof of~\cref{prop:balanced}]		
		By~\cref{lem:redbalanced} there exists $p:V(G)\to\mathbb{R}$ and a red/blue colouring $R\cup B$ of $E(G)$ such that $R$ is triangle-free, \cref{e:cond1} is satisfied, and the red-balanced property is satisfied.
		By~\cref{lem:translate}, $\phi_p$ is triangle-covered with $\sum_{e\in E(G)}\phi(e)=\sum_{e\in E(G)}\phi_p(e)$. 
		Let $\psi:E(G)\to \mathbb{R}$ be defined by
		\[\psi(e):=\begin{cases}
			\phi_p(e) + \frac{\varepsilon}{2e(G)}&\text{ if }e\in B,\\
			\phi_p(e) - \frac{\varepsilon}{4e(G)}&\text{ if }e\in R.
		\end{cases}\]
		
		As a consequence of \cref{e:cond1} and the red-balanced property, $\psi$ is balanced and nonzero. As $R$ is triangle-free, every triangle $T$ of $G$ contains at least one blue edge and at most two red edges, therefore 
		\[\sum_{e\in T}\psi(e)\geq\sum_{e\in T}\phi_p(e)\geq 0,\]
		hence $\psi$ is triangle-covered. By definition of~$\psi$, 
		\[\sum_{e\in E(G)}\psi(e)
		\leq \sum_{e\in E(G)}\phi_p(e) + e(G)\cdot \frac{\varepsilon}{2e(G)}
		\leq \sum_{e\in E(G)}\phi_p(e)+\frac{\varepsilon}{2},\]
		and
		\[\sum_{e\in E(G)}\psi(e)
		\geq \sum_{e\in E(G)}\phi_p(e) - e(G)\cdot \frac{\varepsilon}{4e(G)}
		\geq \sum_{e\in E(G)}\phi_p(e)-\frac{\varepsilon}{4}.\]

		We now perturb $\psi$ again to obtain an untied version of the edge-weighting. We define an ordering of $E(G)=\{e_1,\ldots,e_m\}$, such that, for some $k\in [m-1]$, we have
		\begin{align*}
			\psi(e_i) \geq \psi(e_{i+1})>0&\text{ for all }i\in[k-1], \\
			\psi(e_i) \leq \psi(e_{i+1})<0&\text{ for all }i\in\{k+1,\ldots,m-1\},
		\end{align*}
		breaking ties arbitrarily. Informally, $E(G)$ is ordered with first all positive edges, from largest to smallest weights, and then all negative edges, from most negative to least negative. Let $\nu\in(0,1/2)$ be such that $\nu < \epsilon/(2e(G))$. Let $\psi'$ be defined by, for every $i\in[m]$,
		\[\psi'(e_i):= \psi(e_i)+{\rm sign}(\psi(e_i))\cdot\nu^i
		\]
		Observe that $\psi'$ can be viewed as a {\em stretching} of $\psi$ away from 0 over the real line, where the largest elements (in absolute value) are stretched the most, and where positive edges are stretched more than negative ones. We note first that $\psi'$ is non-zero and that $\psi'(e_i)>0$ if and only if $\psi(e_i)>0$, therefore $\psi'$ is balanced. Also,
        \begin{align*}
        &\psi'(e_i)-\psi'(e_{i+1}) = \psi(e_i)-\psi(e_{i+1})+\nu^i-\nu^{i+1}>0 \text{ for all } i \in [k-1], \\
        &\psi'(e_k)-\psi'(e_{k+1}) = \psi(e_k)-\psi(e_{k+1})+\nu^k+\nu^{k+1}>0, \text{ and} \\
        &\psi'(e_{i+1})-\psi'(e_{i}) = \psi(e_{i+1})-\psi(e_{i})+\nu^i-\nu^{i+1}>0 \text{ for all } i \in \{k+1, \ldots, m-1\}, 
        \end{align*}
        meaning that $\psi'$ is untied. Let $T$ be a triangle of $G$ with edges $e_i$, $e_j$, $e_\ell$ with $i < j < \ell$. Since $R$ has no red triangle, $\psi(e_i) > 0$. Therefore, since $\psi$ is triangle-covered,
        \[
        \sum_{e \in T} \psi'(e) \geq \sum_{e \in T} \psi(e)+\nu^i-\nu^j-\nu^k \geq \sum_{e \in T} \psi(e)+\nu^i(1-\nu-\nu^2) \geq 0,
        \]
		hence $\psi'$ is triangle-covered. Finally, 
		\[\abs*{\sum_{e\in E(G)}\psi'(e)-\sum_{e\in E(G)}\phi(e)}
		\leq 
		\abs*{\sum_{e\in E(G)}\psi'(e)-\sum_{e\in E(G)}\psi(e)} + 
		\abs*{\sum_{e\in E(G)}\psi(e)-\sum_{e\in E(G)}\phi(e)}
		\leq \nu\cdot e(G) + \frac{\varepsilon}{2}
		\leq\varepsilon,\]
		since $\nu < \varepsilon/(2e(G))$.
	\end{proof}
	
	\cref{lem:NoFrac} and \cref{prop:balanced} together yield the following corollary. 
	
	\begin{cor}\label{lem:FarkasExtended} A $K_3$-divisible tripartite graph $G$ does not admit a fractional $K_3$-decomposition if and only if there exists a triangle-covered, totally-negative, balanced, untied and non-zero edge-weighting of $G$. 
	\end{cor}
    \begin{proof}
        The backward direction is immediate from \cref{lem:NoFrac}. For the forward direction, suppose $G$ does not admit a fractional $K_3$-decomposition. Then by \cref{lem:NoFrac}, there exists a triangle-covered and totally-negative edge-weighting $\phi$ of $G$. Thus, by \cref{prop:balanced} with, say, $\varepsilon = -\sum_{e\in E(G)}\phi(e)/2$, there exists a triangle-covered, totally-negative, balanced, untied, and nonzero edge-weighting of $G$.
    \end{proof}

\section{Non-Uniform Proportional Discharging}\label{sec:main}

For simplicity, in this section, we denote by $(U,V,W)$ the vertex set of a tripartite graph $G$ instead of $(V_1,V_2,V_3)$. For such a graph $G$ and a given edge-weighting $\phi:E(G)\to\mathbb{R}$, we colour an edge $e$ of $G$ blue if $\phi(e)\geq 0$, and red otherwise. Furthermore, if $G$ is $K_3$-divisible and $\phi$ is balanced, we define for every vertex $u\in U$ its {\em red degree} as 
\[d_u:=\size*{\set*{v\in V\colon \phi(uv)<0}}=\size*{\set*{w\in W\colon \phi(uv)<0}},\]
and similarly for every $v\in V$ and $w\in W$.
We also label the triangles in $G$ as positive, quasi-positive, quasi-negative, or negative, depending on their colour pattern, as illustrated in~\cref{fig:NameTriangles}.
\begin{table}[H]
    \centering
    \begin{tabular}{cccc}
         \ColoredTriangle{1}{blue}{blue}{blue}&
         \ColoredTriangle{1}{red, dashed}{blue}{blue}&
         \ColoredTriangle{1}{blue}{red, dashed}{red, dashed}&
         \ColoredTriangle{1}{red, dashed}{red, dashed}{red, dashed}
         \\
         Positive & Quasi-positive & Quasi-negative & Negative \\\\
         \multicolumn{2}{c}{\ColoredEdge{1}{blue} Blue edges}&
         \multicolumn{2}{c}{\ColoredEdge{1}{red, dashed} Red edges}
    \end{tabular}
    \caption{Naming convention for triangles.}
    \label{fig:NameTriangles}
\end{table}

The starting point for our discharging rule is a first-order proportional discharging rule developed by Delcourt and Postle in their work on the Nash-Williams conjecture~\cite{DP26NW}, and explained to us in more detail through personal communication during the development of their article. The discharging rule defined hereafter retains this underlying mechanism, but allows the proportion of weight routed through each endpoint to depend on the local configuration. 

\begin{definition}[$f$-proportional discharging]\label{def:NonUniDischarging}
    For every positive integer $n$, let $G$ be a $K_3$-divisible tripartite graph with vertex set $(U,V,W)$ such that $|U|=|V|=|W|=n$. Let $\phi$ be a balanced edge-weighting of $G$, and let $f:[0,1]\to[0,1/2]$ be a non-increasing function. For every edge $e\in E(G)$, colour $e$ red if $\phi(e)<0$ and blue otherwise. For every blue edge $b=vw$, let $\lambda_b\in[0,1/2]$ be defined as
    \[\lambda_{b}=\begin{cases}
        \frac12&\text{ if }d_v=d_w,\\
        f\parens*{\frac1n\cdot\max\{d_v,d_w\}}&\text{ otherwise}.
    \end{cases}\]
    The {\em $f$-proportional discharging of $\phi$} is the following transfer of weights from every blue edge $b$ to its red neighbours (see also~\cref{tab:NonUniDischarging}). For each quasi-positive or quasi-negative triangle $T\in\mathcal{T}$ containing the blue edge $b=vw$ and the red edge $r=uv$, we transfer the weight ${\rm dis}^f_{\phi}(b,r)$ from $b$ to $r$ defined as 
    \[{\rm dis}^f_{\phi}(vw,uv):=
    \begin{cases}
        \dfrac{\lambda_b\cdot\phi(vw)}{d_v} &\text{ if $T$ is quasi-positive and }d_v\geq d_w,\\\\
        \dfrac{(1-\lambda_b)\cdot\phi(vw)}{d_v} &\text{ if $T$ is quasi-positive and }d_v\leq d_w,\\\\
        \dfrac{\lambda_b\cdot |\phi(uv)|}{\max\{d_v,d_w\}}+\dfrac{(1-\lambda_b)\cdot|\phi(uv)|}{\min\{d_v,d_w\}} &\text{ if $T$ is quasi-negative}.
    \end{cases}\]
    For simplicity, we extend the definition of ${\rm dis}^f_{\phi}$ to all pairs of edges by setting ${\rm dis}^f_{\phi}(e_1,e_2):=0$ if $\phi(e_1)<0$, $\phi(e_2)\geq0$, or $e_1,e_2$ are not in a common triangle.
    For a given red edge $r\in E(G)$, we define $\Delta_{\phi}^f(r)$ to be the total weight received by the edge $r$ during discharging, that is,
    \[\Delta_{\phi}^f(r):=\sum_{b\in E(G)}{\rm dis}^f_{\phi}(b,r).\]
\end{definition}

We note that the discharging is said to be {\em proportional} because each positive edge distributes its weight in proportion to the relevant red degrees, and {\em non-uniform} because these proportions vary with the local red degree configuration rather than being constant across all incident red edges; we initially worked on a {\em uniform} proportional discharging, defined by setting $\lambda_b=\frac12$ for all blue edges $b$.

\begin{table}[ht]
    \centering
    \begin{tabular}{>{\centering\arraybackslash}p{0.3\textwidth} |>{\centering\arraybackslash}p{0.3\textwidth} | >{\centering\arraybackslash}p{0.4\textwidth}}
    \multicolumn{2}{c|}{\textbf{$T$ is quasi-positive}}& \textbf{$T$ is quasi-negative}\\\hline
    \DischargingQuasiPos{1} & \DischargingQuasiPosBis{1} &\DischargingQuasiNeg{1} \\
    ${\rm dis}^{f}_{\phi}(b,r):=\dfrac{\lambda_b\cdot \phi(b)}{d_v}$     &
    ${\rm dis}^{f}_{\phi}(b,r):=\dfrac{(1-\lambda_b)\cdot\phi(b)}{d_w}$     
    & ${\rm dis}^{f}_{\phi}(b,r_i):=\dfrac{\lambda_b\cdot|\phi(r_i)|}{d_v}+\dfrac{(1-\lambda_b)\cdot|\phi(r_i)|}{d_w}$
    \end{tabular}
    \caption{Discharging rules when $d_v\geq d_w$}
    \label{tab:NonUniDischarging}
\end{table}

We further note that ${\rm dis}^f_{\phi}(vw,uv)$ in \cref{def:NonUniDischarging} is well-defined since when $d_v=d_w$, $\lambda_b = 1/2 = 1-\lambda_b$. We first prove that during discharging, a positive edge never sends more than its own weight to its neighbours.

\begin{prop}\label{prop:BlueStillPositive}
    Let $n$ be a positive integer, let $G$ be a $K_3$-divisible tripartite graph with vertex set $(U,V,W)$ with $|U|=|V|=|W|=n$, let $f:[0,1]\to[0,1/2]$ be a non-increasing function, and let $\phi:E(G)\to\mathbb{R}$ be a triangle-covered and balanced edge-weighting of $G$. For every edge $b$ such that $\phi(b)\geq 0$, we have
    \[\phi(b)-\sum_{r\in E(G)}{\rm dis}^f_{\phi}(b,r)\geq0.\]
\end{prop}

\begin{proof} Let $b=vw \in E(G)$ be such that $\phi(b)\geq 0$. Without loss of generality, assume that $d_v\geq d_w$, $v \in V$, and $w \in W$. Denote by $\mc{T}^-_b$ the family of quasi-negative triangles containing $b$. If $d_v+d_w=0$, then $b$ is in no quasi-negative and no quasi-positive triangles, hence $\sum_{r\in E(G)}{\rm dis}^f_{\phi}(b,r)=0$. If $d_v>d_w=0$, then there is no red edge incident with $b$ at $w$, and $\mc{T}^-_b=\emptyset$. Hence 
\[\sum_{r\in E(G)}{\rm dis}^f_{\phi}(b,r)\leq d_v\frac{\lambda_b\cdot\phi(b)}{d_v}\leq \phi(b).\]
Assume now that $d_v\geq d_w>0$. We have 
\begin{align*}
    \sum_{r\in E(G)}{\rm dis}^f_{\phi}(b,r)
 \leq 
 &\sum_{u\in U\colon uvw\in \mc{T}^-_b} \brackets*{
 \frac{\lambda_{b}\cdot|\phi(uv)|}{d_v}+\frac{(1-\lambda_{b})\cdot|\phi(uv)|}{d_w}
 +\frac{\lambda_{b}\cdot|\phi(uw)|}{d_v}+\frac{(1-\lambda_{b})\cdot|\phi(uw)|}{d_w}
 }\\
 &+ \parens*{d_v-|\mc{T}^-_b|}\cdot \frac{\lambda_{b}\cdot\phi(vw)}{d_v} 
 + \parens*{d_w-|\mc{T}^-_b|}\cdot \frac{(1-\lambda_{b})\cdot\phi(vw)}{d_w}\\
 \leq &\sum_{u\in U\colon uvw\in \mc{T}^-_b} \brackets*{
 \frac{\lambda_{b}\cdot\phi(vw)}{d_v}+\frac{(1-\lambda_{b})\cdot\phi(vw)}{d_w}
 }\\
 &+ \parens*{d_v-|\mc{T}^-_b|}\cdot \frac{\lambda_{b}\cdot\phi(vw)}{d_v}
 + \parens*{d_w-|\mc{T}^-_b|}\cdot \frac{(1-\lambda_{b})\cdot\phi(vw)}{d_w} \\
 = & \,\,  \phi(vw),
\end{align*}
as desired, where for the second inequality we have used the fact that $\phi$ is triangle-covered.
\end{proof}

Next, we lower bound the amount of weight received by a negative edge during discharging. For $a \in \mathbb{R}$, we define $a^+ := \max(a, 0)$ and for convenience, we extend the definition of the supremum so that the supremum of any function over the empty set is $0$. 

\begin{prop}\label{prop:dischargingvalue}
   Let $n$ be a positive integer and let $G$ be a $K_3$-divisible tripartite graph with vertex set $(U,V,W)$ with $|U|=|V|=|W|=n$ and $\hat\delta(G)\geq(1-d)n$ for some $d\in[0,1/2)$. Let $f:[0,1]\to[0,1/2]$ be a non-increasing function, and let $\phi$ be a nonzero, triangle-covered, and balanced edge-weighting of $G$. Then, for every edge $r=uv\in G$ with $\phi(r)<0$ and $d_u\geq d_v$, we have
    \[\frac{\Delta^f_{\phi}(r)}{|\phi(r)|}\geq 
    \brackets*{M(y,1-x)-\frac{f(x)}{x}}\cdot (x-d)^+
    +\brackets*{M(x,1-y)-\frac{f(x)}{x}}\cdot (y-d)^+
    +\frac{f(x)}{x}\cdot (1-2d),\]
where $x:=\frac{d_u}{n}$, $y:=\frac{d_v}{n}$ satisfy $0<y\leq x<1$ and $0<x+y\leq 1$, and
\begin{align*}
    M:(0,1)^2&\to\mathbb{R}\\
    (a,b) &\mapsto  \frac{1}{a}-\sup_{a<z\leq b}\brackets*{f(z)\cdot\parens*{\frac{1}{a}-\frac{1}{z}}}.
\end{align*}
\end{prop}

\begin{proof}
For each edge $e \in E(G)$, colour $e$ blue if $\phi(e)\geq 0$, and red if $\phi(e)< 0$. Let 
    \[ R = \{e\in E(G)\colon \phi(e)<0\},\text{ and }\quad B = \{e\in E(G)\colon \phi(e)\geq 0\}.\]

We use the $f$-proportional discharging of~\cref{tab:NonUniDischarging} to transfer weights between edges. Fix a red edge $r=uv\in R$, with $u\in U$ and $v\in V$, and recall that $\Delta^f_{\phi}(r)$ denotes the sum of weights received by $r$ during the discharging, that is
\[\Delta^f_{\phi}(r)=\sum_{b\in B}{\rm dis}^f_{\phi}(b,r).\]

Let $x:=\frac{d_u}{n}$ and $y:=\frac{d_v}{n}$. Observe that $d_u,d_v\geq 1$ (because $u,v$ are at least adjacent to the red edge $uv$), and that $d_u+d_v\leq n$ because $G$ contains no negative triangle. We partition the common neighbours of $u,v$ into $3$ sets of vertices, depending on the red neighbourhood of $u$ and $v$.

\begin{minipage}{\linewidth}
  \centering
  \begin{minipage}{0.40\linewidth}
    \begin{align*}
        W^-_u&:=\big\{w \in W\colon uvw\in\mc{T}(G),~uw\in R,~vw\in B\big\},\\
        W^-_v&:=\big\{w \in W\colon uvw\in\mc{T}(G),~uw\in B,~vw\in R\big\},\\
        W^+&:=\big\{w \in W\colon uvw\in\mc{T}(G),~uw\in B,~vw\in B\big\}.
    \end{align*}
  \end{minipage}
  \hfill
  \begin{minipage}{0.5\linewidth}
        \begin{figure}[H]
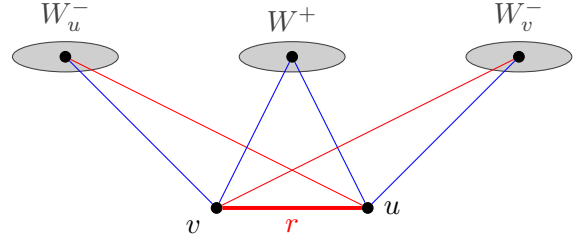

            \centering
            \RedEdges{2}
            \caption{Common neighbours of $u,v$ in $G$.}
            \label{fig:RedNeighbors}
        \end{figure}
  \end{minipage}
\end{minipage}\bigskip

For every vertex $w\in W^+$, that is, every $w \in W$ such that $T=uvw$ is a quasi-positive triangle in $G$, let $z:=d_w/n$. We have
\[\frac{{\rm dis}^f_{\phi}(vw,uv)}{\phi(vw)}\cdot d_v=
\begin{cases}
    f(y),&\text{ if }y>z,\\
    1/2\geq f(y),&\text{ if }y=z,\\
    1-f(z)\geq 1/2\geq f(y),&\text{ if }y<z,
\end{cases}
\]
hence
\[
{\rm dis}^f_{\phi}(vw,uv)\geq\frac{f(y)}{d_v}\cdot \phi(vw).\]
Similarly,
\[{\rm dis}^f_{\phi}(uw,uv)\geq\frac{f(x)}{d_u}\cdot \phi(uw).\]
Recall that $\phi$ is triangle covered, therefore we have $\phi(vw)+\phi(uw)\geq |\phi(r)|$. As $d_u\geq d_v$ and $f$ is non-increasing, we obtain that, for every $w\in W^+$, the weight edge $r$ received during discharging from the triangle $uvw$ is at least
\[ \frac{f(y)}{d_v}\cdot \phi(vw) + \frac{f(x)}{d_u}\cdot \phi(uw)
\geq \frac{f(x)}{d_u}\cdot \parens*{\phi(vw)+\phi(uw)} 
\geq \frac{f(x)}{d_u}\cdot |\phi(r)|.\]

We now fix a vertex $w\in W^-_v$, that is, a $w \in W$ such that $uvw$ is a quasi-negative triangle in $G$ with $b:=uw\in B$ and $vw\in R$, and we again define $z:=d_w/n$. We have
\[
\frac{{\rm dis}^f_{\phi}(b,r)}{|\phi(r)|}
=\frac{{\rm dis}^f_{\phi}(uw,uv)}{|\phi(uv)|}
=\begin{cases}
    \dfrac{f(x)}{d_u}+\dfrac{1-f(x)}{d_w},& d_w< d_u,\\
    \dfrac{1}{d_u},&d_w=d_u\\
    \dfrac{1-f(z)}{d_u}+\dfrac{f(z)}{d_w},& d_w> d_u.
\end{cases}
=\begin{cases}
    \dfrac{f(x)}{d_u}+\dfrac{1-f(x)}{d_w},& d_w\leq d_u,\\~\\
    \dfrac{1-f(z)}{d_u}+\dfrac{f(z)}{d_w},& d_w> d_u.
\end{cases}\]
When $d_w \leq d_u$, this is minimised when $d_w=d_u$, with value $1/d_u$. When $d_w>d_u$, this value is $\frac{1}{d_u}-f(z)\cdot\parens*{\frac{1}{d_u}-\frac{1}{d_w}}$.
As $\phi$ is triangle-covered, we know that $G$ contains no negative triangles (i.e. fully red triangles), hence $d_v+d_w\leq n$. It follows that
\[\frac{{\rm dis}^f_{\phi}(b,r)}{|\phi(r)|} \geq \frac{1}{n}\cdot\parens*{
\frac{1}{x}-\sup_{x<z\leq 1-y}\brackets*{f(z)\cdot\parens*{\frac{1}{x}-\frac{1}{z}}
}}=M(x,1-y)\cdot\frac1n.\]
Similarly,
\[
\frac{{\rm dis}^f_{\phi}(vw,uv)}{|\phi(uv)|} \geq M(y, 1-x)\cdot\frac1n
\]
for all $w \in W_u^-$. We obtain that the edge $r$ receives, from the discharging process, a total weight of $\Delta^f_{\phi}(r)$ with
\begin{align*}
\frac{\Delta^f_{\phi}(r)}{|\phi(r)|}\geq 
    M(y,1-x)\cdot\frac{|W^-_u|}{n}
    +M(x,1-y)\cdot\frac{|W^-_v|}{n}
    +\frac{f(x)}{x}\cdot \frac{|W^+|}{n}.
\end{align*}
As $\hat{\delta}(G)\geq (1-d)n$, we have
\((1-2d)n\leq |W^+|+|W_u^-|+|W_v^-|,\)
and therefore
\[\frac{\Delta^f_{\phi}(r)}{|\phi(r)|}\geq 
    \brackets*{M(y,1-x)-\frac{f(x)}{x}}\cdot \frac{|W^-_u|}{n}
    +\brackets*{M(x,1-y)-\frac{f(x)}{x}}\cdot \frac{|W^-_v|}{n}
    +\frac{f(x)}{x}\cdot (1-2d).\]
Observe that  $M(a,b)\geq \frac{1}{2a}$ for all $a, b\in(0,1)$, hence
\[M(y,1-x)\geq\frac{1}{2y}\geq \frac{1}{2x}\geq \frac{f(x)}{x},\quad\text{ and}\qquad
M(x,1-y)\geq\frac{1}{2x}\geq \frac{f(x)}{x},\]
and as $G$ contains no negative triangle, and using the fact that $\phi$ is balanced, we have
\[|W^-_u|\geq (d_u-dn)^+=(x-d)^+n,\text{ and}\qquad |W^-_v|\geq (d_v-dn)^+=(y-d)^+n.\]
Combining these inequalities, we obtain
\begin{align*}
\frac{\Delta^f_{\phi}(r)}{|\phi(r)|}\geq 
    &\brackets*{M(y,1-x)-\frac{f(x)}{x}}\cdot (x-d)^+\\
    +&\brackets*{M(x,1-y)-\frac{f(x)}{x}}\cdot (y-d)^+\\
    +&\frac{f(x)}{x}\cdot (1-2d),
\end{align*}
as desired.
\end{proof}

It remains to prove that, for well-chosen $d$ and $f$, the lower bound obtained in \cref{prop:dischargingvalue} is always at least one, for all possible values of $x,y$.

\begin{prop}\label{prop:optim}
    There exist a non-increasing function $f:[0,1]\to[0,\frac12]$ and $\eta>0$ such that the following holds for every $0\leq d\leq 0.231+\eta$. Let
    \begin{align*}
        g_d:(0,1)^2&\to \mathbb{R}\\
        (x,y)&\mapsto     \brackets*{M(y,1-x)-\frac{f(x)}{x}}\cdot (x-d)^+
    +\brackets*{M(x,1-y)-\frac{f(x)}{x}}\cdot (y-d)^+
    +\frac{f(x)}{x}\cdot (1-2d),
    \end{align*}
    where 
    \begin{align*}
        M:(0,1)^2&\to\mathbb{R}\\
        (a,b) &\mapsto  \frac{1}{a}-\sup_{a<z\leq b}\brackets*{f(z)\cdot\parens*{\frac{1}{a}-\frac{1}{z}}}.
    \end{align*}
    Then $g_d(x,y)\geq 1$ for all $x\geq y \in (0,1)$ with $0<x+y\leq 1$.
\end{prop}

The proof of \cref{prop:optim} consists of straightforward but somewhat lengthy numerical analysis and thus has been deferred to \cref{sec:AppendixOptim}. Here we just record that the function playing the role of $f$ in our proof is the piecewise-linear function defined by
\begin{equation}\label{e:f}
f(t):=
\begin{cases}
\frac12&\text{ if } t\leq \frac25,\\
0&\text{ if } t>\frac{7}{10},\\
\frac{7-10t}{6}&\text{ otherwise. }
\end{cases}
\end{equation}

\subsection{Main proof}

We are finally ready to prove our main theorem.

\begin{proof}[Proof of~\cref{thm:main}.]
Let $f:[0,1]\to[0,1/2]$ and $\eta>0$ be given by~\cref{prop:optim}. Let $G$ be a $K_3$-divisible tripartite graph with $\hat\delta(G)\geq(1-d)n$ for $d=0.231+\eta$. Assume that $G$ contains no fractional $K_3$-decomposition. Therefore, by~\cref{lem:FarkasExtended}, there exists a triangle-covered, totally-negative, balanced and non-zero edge-weighting $\phi:E(G)\to\mathbb{R}$.

Let $\psi:E(G)\to\mathbb{R}$ be the edge-weighting obtained from $\phi$ after applying the discharging ${\rm dis}^f_{\phi}$ as defined in~\cref{def:NonUniDischarging}. By~\cref{prop:BlueStillPositive}, for every edge $b$ of $G$ with $\phi(b)\geq0$ we have
\[\psi(b) = \phi(b)-\sum_{r\in E(G)}{\rm dis}^f_{\phi}(b,r)\geq 0.\]
By~\cref{prop:dischargingvalue,prop:optim} for every edge $r\in G$ with $\phi(r)<0$, we have
\[\psi(r) = \Delta^f_{\phi}(r)+\phi(r)\geq 0.\]
It follows that 
\[\sum_{e\in E(G)}\phi(e)=\sum_{e\in E(G)}\psi(e)\geq 0,\]
contradicting the fact that $\phi$ is totally-negative. Therefore $G$ admits a fractional $K_3$-decomposition.
\end{proof}

\section{Lower bound}\label{sec:LowerBound}

Here we prove~\cref{prop:LowerBound}. 
Recall that a {\em $k$-plex} is a partial Latin square of order $n$ containing $kn$ entries such that every row, column, and symbol is used exactly $k$ times. A {\em transversal} of a Latin square is a $1$-plex. Wanless~\cite[Theorem 1]{W2002} proved that for every $n>2$ and every $0\leq k\leq n$, there exists a completable $k$-plex, and for $k>n/4$ that not all $k$-plexes are completable~\cite[Theorem 2]{W2002}. The following construction comes from ``Case 1'' in the proof of~\cite[Theorem 2]{W2002}, with a proof adapted for fractional decompositions.

\begin{proof}[Proof of \cref{prop:LowerBound}]
Let $k:=n/4\geq 2$ and $m:=n/2=2k$. Let $H_1$ be a completable $k$-plex of order $m$ (such a completable partial Latin square exists by~\cite[Theorem~1]{W2002}), interpreted as a subgraph of $K_{m, m, m}$ on vertex set $(U_1,V_1,W_1)$. Similarly, let $H_2\subseteq K_{m,m,m}$ be a $k$-plex of order $m$ on vertex set $(U_2,V_2,W_2)$, disjoint from $(U_1,V_1,W_1)$. Let $H\subseteq K_{n,n,n}$ be the $n/4$-plex defined as the disjoint union of $H_1,H_2$, on vertex set $(U_1\cup U_2,V_1\cup V_2,W_1\cup W_2)$. Let $G:=K_{n,n,n}\setminus H$. We say that an edge $xy\in E(G)$ is {\em internal} if $x,y\in V(H_i)$ for some $i\in\{1,2\}$, and {\em crossing} otherwise. 

Let $\phi:E(G)\to \mathbb{R}$ be defined by
\[\phi(e):=\begin{cases}
    +2&\text{ if $e$ is internal,}\\
    -1&\text{ if $e$ is crossing.}
\end{cases}\]

Observe first that every triangle $T$ of $G$ contains either 1 or 3 internal edges, hence that
\[\sum_{e\in T} \phi(e)\in\{ 0,+6\}\]
and $\phi$ is triangle-covered. Note that $G$ contains $6m^2=24k^2$ crossing edges, and $6m(m-k)=12k^2$ internal edges, hence
\[\sum_{e\in E(G)}\phi(e)=0.\]
Finally, let $G'=G-T$ where $T$ is any triangle of $G$ with vertex set fully in $H_1$. Obverse that such a triangle $T$ exists as $H_1$ is completable. We have  $\hat\delta(G')=\frac{3n}{4}-1$, and since $H$ is a triangle packing, $G=K_{n,n,n}\setminus H$ is $K_3$-divisible; deleting the triangle $T$ then preserves $K_3$-divisibility. Let $\phi'$ be the edge-weighting of $G'$ induced by $\phi$. Then $\phi'$ is triangle-covered and
\[\sum_{e\in E(G')}\phi'(e)=\sum_{e\in E(G)}\phi(e)-\sum_{e\in T}\phi(e)=0-6<0,\]
so $\phi'$ is totally negative. By~\cref{lem:NoFrac}, $G'$ admits no fractional $K_3$-decomposition.
\end{proof}

\section{Future work}\label{sec:FutureWork}

In this section, we discuss the potential to use different discharging rules to improve on \cref{thm:main}. We also discuss some related future research directions.

The first discharging rule that we used to attack \cref{conj:decompo} was a uniform proportional discharging rule. That is, the discharging rule from \cref{def:NonUniDischarging} with $f$ defined by $f(x)=1/2$, the first-order proportional discharging considered by Delcourt and Postle in their work on the Nash-Williams conjecture~\cite{DP26NW}. This yielded \cref{thm:main} with $d=0.215$ in place of $d=0.231$. We then switched to $f$-uniform proportional discharging and experimented with different choices for $f$. Settling on $f$ defined by \cref{e:f} yielded \cref{thm:main}. There may be some room for improving \cref{thm:main} just by changing the choice of $f$, but the following proposition tells us that no choice of $f$ could produce an improvement better than $d=4/17\approx 0.235$.

\begin{prop}\label{prop:nonUniformUB}
    For every non-increasing function $f:[0,1]\to[0,\frac12]$, the following holds for every $d> 4/17$. Let
    \begin{align*}
        g_d:(0,1)^2&\to \mathbb{R}\\
        (x,y)&\mapsto     \brackets*{M(y,1-x)-\frac{f(x)}{x}}\cdot (x-d)^+
    +\brackets*{M(x,1-y)-\frac{f(x)}{x}}\cdot (y-d)^+
    +\frac{f(x)}{x}\cdot (1-2d),
    \end{align*}
    where 
    \begin{align*}
        M:(0,1)^2&\to\mathbb{R}\\
        (a,b) &\mapsto  \frac{1}{a}-\sup_{a<z\leq b}\brackets*{f(z)\cdot\parens*{\frac{1}{a}-\frac{1}{z}}}.
    \end{align*}
    Then $g_d(x,y)< 1$ for some $x\geq y \in (0,1)$ with $0<x+y\leq 1$.
\end{prop}
\begin{proof}
    Fix a non-increasing function $f:(0,1)\to[0,1/2]$.
    If $d\geq \frac13$, let $x^*=y^*=d\geq\frac13$ and observe that
    \[g_d(x^*,y^*)\leq 3\cdot f\parens*{\frac13}\cdot(1-2d)\leq f\parens*{\frac13}< 1.\]
    Assume now that $d< \frac13$. Let $y=d<x<1/2$. We have 
    \[
    g_d(x,d):=\brackets*{M(d,1-x)-\frac{f(x)}{x}}\cdot (x-d) +\frac{f(x)}{x}\cdot (1-2d).
    \]
    Because $d<x<1-x$ we have
    \[M(d,1-x) \leq  \frac{1}{d}-\brackets*{f(x)\cdot\parens*{\frac{1}{d}-\frac{1}{x}}
    }=\frac{1-f(x)}{d}+\frac{f(x)}{x}
    ,\]
    and
    \[g_d(x,d)\leq (1-f(x))\frac{x-d}{d}+f(x)\frac{1-2d}{x}=\frac{x-d}{d}+f(x)\parens*{\frac{1-2d}{x}-\frac{x-d}{d}}.\]
    Let $x^*:=\sqrt{d(1-2d)}\leq 1-d$, hence such that 
    \[\frac{1-2d}{x^*}-\frac{x^*-d}{d} = 1.\]
    Observe that $d< \frac13$ ensures that $d< x^*$, hence $(x^*,d)$ is an admissible input to $g_d$. We then have
    \[g_d(x^*,d)\leq \frac{x^*-d}{d}+f(x^*).\]
    Using $f(x)\leq1/2$ we obtain,
    \[g_d(x^*,d)\leq \sqrt{\frac{1-2d}{d}}-\frac12,\]
    and for $d>4/17$ we have
    \[g_d(x^*,d)\leq \sqrt{\frac{1-2d}{d}}-\frac12< 1,\]
    as required.
\end{proof}

In light of \cref{prop:nonUniformUB}, to use our framework to improve \cref{thm:main}, one must either improve \cref{prop:dischargingvalue} for $f$-proportional discharging, or use entirely new discharging rules. A natural direction for improvement is to try to use the discharging rules used by Delcourt and Postle~\cite{DP26NW}, which we define now. Let $G$ be a graph and let $\phi$ be a triangle-covered, non-zero, and untied edge-weighting of $G$. Let $r=uv \in E(G)$ with $\phi(r)<0$ and let $b = vw \in E(G)$ with $\phi(b)>0$ be such that $r$, $b$ and $b' := uw$ form a triangle $T$ in $G$. Define the demand function $\dem$ by 
	\[
	\dem(b, r) = \begin{cases}
		|\phi(r)| & \text{if } T \text{ is quasi-negative},\\
		|\phi(r)|\frac{\phi(b)}{\phi(b)+\phi(b')} & \text{if } |\phi(r)| > \max\{\phi(b), \phi(b')\} \text{ and } T \text{ is quasi-positive},\\
		|\phi(r)| & \text{if } \phi(b) \geq \max\{\phi(b'), |\phi(r)|\} \text{ and } T \text{ is quasi-positive},\\
		0 & \text{if } \phi(b') \geq \max\{\phi(b), |\phi(r)|\} \text{ and } T \text{ is quasi-positive}.\\
	\end{cases}
	\]
	The discharging function $\disDP$ used by Delcourt and Postle is defined by
	\[
	\disDP(b, r) = \phi(b)\frac{\dem(b, r)}{\sum_{r' \in \mathcal{T}^-(b)} \dem(b, r')},
	\]
	where $\mathcal{T}^-(b)$ is the set of edges of negative weight that are in a triangle with $b$. So a positive edge splits its weight between all negative edges it is involved in a triangle with, where the weight such an edge gets is proportional to how much it `demands' from $b$. 
	
	The analysis of the discharging rule $\disDP$ is much more involved than the analysis of the $f$-proportional discharging defined in \cref{def:NonUniDischarging}. When we trialled this approach in the partite setting, we encountered some hypothetical configurations that we could not rule out where some negative edges remained negative after applying the discharging rules. For example, suppose that $G$ is a $K_3$-divisible tripartite graph with tripartition $(U, V, W)$ where $|U|=|V|=|W|=n \equiv 0 \bmod 4$ and $\hat\delta(G) \geq 3n/4$. Suppose that $\phi$ is a non-zero, triangle-covered edge-weighting of $G$ and colour edges blue or red depending on whether they have positive or negative weight, respectively. Let $r = uv \in E(G)$ be a red edge with weight $-a$ where $u \in U$ and $v \in V$. Suppose that we can partition $W = W_1 \cup W_2 \cup W_3 \cup W_4$ into four equal parts so that:
	\begin{itemize}
		\item The neighbourhood of $u$ is $W_1 \cup W_3 \cup W_4$,
		\item Each edge $uw$ with $w \in W_3$ is blue and each edge $uw$ with $w \in W_4$ is red,
		\item The neighbourhood of $v$ is $W_2 \cup W_3 \cup W_4$,
		\item Each edge $vw$ with $w \in W_3$ is red and each edge $vw$ with $w \in W_4$ is blue.
	\end{itemize} 
	Then every triangle containing $r$ is quasi-negative. Suppose further that every red edge incident to $r$ has weight $-a$ and every blue edge has weight $2a$. Further suppose that each blue edge incident to $r$ is in exactly $3n/4-1$ triangles other than the one containing $r$, each of which has two edges of weight $2w$ and one of weight $-4w$. Then for each blue edge $b$ incident to $r$, $\dem(b, r) = w$ and $\dem(b, r')=2w$ for all $r' \in \mathcal{T}^-(b) \setminus \{r\}$ and so
	\[
	\disDP(b, r) = 2w\frac{w}{w+(3n/4-1)2w} = \frac{4w}{3n-2}.
	\]
	Since there are $n/2$ choices for $b$, it follows that the total weight that $r$ gains from the discharging process is $2w/(3-2/n)<w$, meaning that $r$ remains red after discharging. Note that technically $\phi$ is not untied, however it is a simple matter to slightly perturb the weights of edges so that $\phi$ is untied and the conclusion remains the same. So in order to use the discharging rule of Delcourt and Postle to resolve the fractional version of \cref{conj:decompo}, configurations such as those described above must be shown to be impossible. However, a straightforward computation shows that using $f$-proportional discharging with, say, $f$ defined by $f(x)=1/2$, the final weight of $r$ in the above configuration is non-negative, meaning that this configuration is not problematic for our discharging rule.\bigskip
	
	We note that if one shows that $\hat\delta_f^*=3/4$, this would still only resolve the conjecture asymptotically, in the sense that it would `only' prove that for any $\eps>0$ and any sufficiently large $n$, if $G$ is a $K_3$-divisible spanning subgraph of $K_{n, n, n}$ with $\hat\delta(G)\geq(3/4+\epsilon)n$, then $G$ admits a $K_3$-decomposition. To remove the $\epsilon$ in the above statement would likely be an enormous effort, similarly to the analogous problem for the Nash-Williams conjecture as achieved by Delcourt and Postle.\bigskip

    It is natural to ask about the thresholds for larger cliques. For an $r$-partite graph $G$ with vertex classes $V_1,\ldots,V_r$, we define its partite minimum degree similarly to the tripartite case presented here, as $\hat\delta(G) := \min\set*{d(v, V_j ) \colon j \in [r], v \in V (G)\setminus V_j}$. We further say that such an $r$-partite graph is $K_r$-divisible if for every $i,j\in [r]$ and $v\in V(G)\setminus(V_i\cup V_j)$, we have $d(v,V_i)=d(v,V_j)$. Let $\hat{\delta}^*_{f,r}$ denote the corresponding threshold for fractional $K_r$-decompositions of sufficiently large balanced $r$-partite graphs, that is, the infimum over all $c$ such that every $K_r$-divisible $r$-partite graph on vertex set $(V_1,\ldots,V_r)$ with $|V_1|=\ldots=|V_r|=n$ for $n$ sufficiently large and $\hat{\delta}(G) \geq cn$ contains a fractional $K_r$-decomposition. Barber, Kühn, Lo, Osthus and Taylor \cite{BKLOT2017} conjectured that $\hat{\delta}^*_{f,r}=1-1/(r+1)$, in analogy with the corresponding conjecture for the non-partite fractional $K_r$-decomposition threshold. Montgomery~\cite{M17Partite} proved the upper bound $\hat{\delta}^*_{f,r} \leq 1-1/(10^6r^3)$, while his best bound in the non-partite setting is $1-1/(100r)$~\cite{M2019}. The partite and non-partite conjectures were subsequently disproved by Delcourt, Henderson, Lesgourgues and Postle~\cite{DHLP25Counter} as they proved that these thresholds are at least $1-1/(c\cdot(r+1))$ for some constant $c>1$. For the non-partite problem, the threshold is known to be of the form $1-\Theta(1/r)$. It is then natural to expect the same order in the partite setting, although no precise conjecture has been formulated. Longbrake, Simkin, Yepremyan and Zheng~\cite{ZhengTalk} recently announced an improvement on Mongomery's upper bound in the partite setting, proving that $\hat{\delta}^*_{f,r} \leq 1-1/(30r^2)$. These questions are closely related to the completion of mutually orthogonal partial Latin squares. 

\paragraph{\textbf{Declaration of AI use:}} ChatGPT 5.6 has been used solely to improve the numerical verification coding of~\cref{app:CodePython,app:CodeMathematica} in the Appendix. There has been no use of any generative AI tools for research or preparation of this article.

\paragraph{\textbf{Acknowledgements:}} This project was started during the First Bulgarian Workshop in Probabilistic and Extremal Combinatorics at the Minu Balkanski Foundation,  funded by FWF 10.55776/ESP624 and FWF 10.55776/ESP3863424. TL thanks Michelle Delcourt and Luke Postle for several helpful discussions about their work on the Nash-Williams conjecture, which provided useful perspective during the development of this project, and for comments on a first draft of this manuscript. KP was supported by ERC Starting Grant ``RANDSTRUCT'' No.~101076777. JA was supported by the Deutsche Forschungsgemeinschaft (DFG, German Research Foundation) under Germany’s Excellence Strategy – The Berlin Mathematics Research Center MATH+ (EXC-2046/2, project ID: 390685689).

\bibliographystyle{plain}
\bibliography{bibliography}

\begin{thebibliography}{10}

\bibitem{AH1983}
Lars~D{\o}vling Andersen and Anthony~JW Hilton.
\newblock Thank {E}vans!
\newblock {\em Proceedings of the London Mathematical Society}, 3(3):507--522, 1983.

\bibitem{BKLOT2017}
Ben Barber, Daniela K{\"u}hn, Allan Lo, Deryk Osthus, and Amelia Taylor.
\newblock Clique decompositions of multipartite graphs and completion of {L}atin squares.
\newblock {\em Journal of Combinatorial Theory, Series A}, 151:146--201, 2017.

\bibitem{B2013}
Padraic Bartlett.
\newblock Completions of $\varepsilon$-dense partial {L}atin squares.
\newblock {\em Journal of Combinatorial Designs}, 21(10):447--463, 2013.

\bibitem{BD2019}
Flora~C Bowditch and Peter~J Dukes.
\newblock Fractional triangle decompositions of dense $3$-partite graphs.
\newblock {\em Journal of Combinatorics}, 10(2):255--282, 2019.

\bibitem{CH1985}
Amanda Chetwynd and Roland H{\"a}ggkvist.
\newblock {\em Completing Partial $n\times n$ Latin Squares where Each Row, Column and Symbol is Used at Most $cn$ Times}.
\newblock University, Dpto. of Mathematics, 1985.

\bibitem{DH1983}
David~E Daykin and Roland H{\"a}ggkvist.
\newblock Completion of sparse partial {L}atin squares.
\newblock {\em Graph theory and combinatorics}, pages 127--132, 1983.

\bibitem{DHLP25Counter}
Michelle Delcourt, Cicely Henderson, Thomas Lesgourgues, and Luke Postle.
\newblock Beyond {N}ash-{W}illiams: Counterexamples to clique decomposition thresholds for all cliques larger than triangles.
\newblock {\em Proceedings of the American Mathematical Society}, 154(12):4107--4125, 2026.

\bibitem{DP2021progress}
Michelle Delcourt and Luke Postle.
\newblock Progress towards {N}ash-{W}illiams' {C}onjecture on triangle decompositions.
\newblock {\em Journal of Combinatorial Theory, Series B}, 146:382--416, 2021.

\bibitem{DP26NW}
Michelle Delcourt and Luke Postle.
\newblock A proof of {N}ash-{W}illiams' conjecture.
\newblock {\em arXiv preprint arXiv:2606.11178}, 2026.

\bibitem{euler}
Leonhard Euler.
\newblock Recherches sur un nouvelle esp\'{e}ce de quarr\'{e}s magiques.
\newblock {\em Verhandelingen uitgegeven door het zeeuwsch Genootschap der Wetenschappen te Vlissingen}, pages 85--239, 1782.

\bibitem{MR1580578}
Julius Farkas.
\newblock Theorie der einfachen {U}ngleichungen.
\newblock {\em J. Reine Angew. Math.}, 124:1--27, 1902.

\bibitem{G1991}
Torbj{\"o}rn Gustavsson.
\newblock {\em Decompositions of large graphs and digraphs with high minimum degree}.
\newblock PhD thesis, Univ., 1991.

\bibitem{HR01}
Penny~E. Haxell and Vojtech R{\"o}dl.
\newblock Integer and fractional packings in dense graphs.
\newblock {\em Combinatorica}, 21(1):13--38, 2001.

\bibitem{HeCheng}
Hengzhi He and Guang Cheng.
\newblock Fractional triangle decompositions at partite minimum degree $4n/5$.
\newblock {\em arXiv preprint arXiv:2609.31758}, 2026.

\bibitem{M17Partite}
Richard Montgomery.
\newblock Fractional clique decompositions of dense partite graphs.
\newblock {\em Combinatorics, Probability and Computing}, 26(6):911--943, 2017.

\bibitem{M2019}
Richard Montgomery.
\newblock Fractional clique decompositions of dense graphs.
\newblock {\em Random Structures \& Algorithms}, 54(4):779--796, 2019.

\bibitem{Nash1970}
C~St~JA Nash-Williams.
\newblock An unsolved problem concerning decomposition of graphs into triangles.
\newblock {\em Combinatorial Theory and its Applications}, 3(1070):1179--1183, 1970.

\bibitem{S1981}
Bohdan Smetaniuk.
\newblock A new construction on {L}atin squares {I}. a proof of the {E}vans conjecture.
\newblock {\em Ars Combin}, 11:155--172, 1981.

\bibitem{W2002}
Ian~M Wanless.
\newblock A generalisation of transversals for {L}atin squares.
\newblock {\em The {E}lectronic {J}ournal of {C}ombinatorics}, pages R12--R12, 2002.

\bibitem{YF2026}
Shikang Yu and Tao Feng.
\newblock On the completion of $\epsilon$-dense partial {L}atin squares.
\newblock {\em arXiv preprint arXiv:2606.31143}, 2026.

\bibitem{ZhengTalk}
Xiangxiang~Michael Zheng.
\newblock The (fractional) {$K_r$}-decomposition threshold for {$r$}-partite graphs.
\newblock Talk at the SIAM Conference on Discrete Mathematics, 2026.
\newblock Joint work with Sean Longbrake, Michael Simkin, and Liana Yepremyan.

\end{thebibliography}

\appendix
\section{Optimisation}\label{sec:AppendixOptim}
\begin{proof}[Proof of~\cref{prop:optim}]
Let $f$ be defined by
\[f(t):=
\begin{cases}
\frac12&\text{ if } t\leq \frac25,\\
0&\text{ if } t>\frac{7}{10},\\
\frac{7-10t}{6}&\text{ otherwise. } \
\end{cases}\]
Write 
\(H_a(z):=f(z)\cdot\parens*{\frac{1}{a}-\frac{1}{z}}.\) 
Observe first that for all $a,b\in(0,1)$,
\begin{equation*}
M(a,b)\geq m(a):=\frac{1}{a}-\sup_{a<z<1}H_a(z).    
\end{equation*}
If $a<z\leq \frac25$, then by definition of $f$ we have
\[H_a(z):=\frac12\cdot\parens*{\frac{1}{a}-\frac{1}{z}},\]
which is increasing in $z$. If $z>\frac7{10}$, we have $H_a(z)=0$. If $z\in\brackets*{\frac25,\frac7{10}}$, we have
\[H_a(z)=\frac{(7-10z)(z-a)}{6az},\]
and on that domain,\[H'_a(z)\geq 0 \iff 7a-10z^2\geq 0.\]
It follows that the maximum of $H_a(z)$ over $(a,1)$ is attained at 
\[z_0:=
\begin{cases}
    \frac25&\text{ if }a\leq\frac{8}{35}\\
    \sqrt{\frac{7a}{10}}&\text{ if }\frac{8}{35}\leq a\leq\frac{7}{10},
\end{cases}\]
while $H_a(z)=0$ for every $z\geq a\geq\frac{7}{10}$.

We obtain that 
\[M(a,b)\geq m(a)= 
\begin{cases}
    \dfrac{1}{2a}+\dfrac54&\text{ if }0<a\leq\frac{8}{35},\\\\
    \dfrac{\sqrt{70}}{3\sqrt{a}}-\dfrac53-\dfrac{1}{6a}  &\text{ if }\frac8{35}\leq a\leq \frac7{10},\\\\
    \dfrac1a&\text{ if }\frac7{10}\leq a\leq 1,
\end{cases}
\]

and we note that $m$ is decreasing in $a$ for future use. It remains to prove that 
\[B_d(x,y):=\brackets*{m(y)-\frac{f(x)}{x}}\cdot (x-d)^+
    +\brackets*{m(x)-\frac{f(x)}{x}}\cdot (y-d)^+
    +\frac{f(x)}{x}\cdot (1-2d)\geq 1\]
for all $0<y\leq x<1$ with $x+y\leq 1$. Observe that $m(a)\geq \frac{1}{2a}\geq\frac{f(a)}{a}$ for all $a\in(0,1)$, therefore $B_d(x,y)$ is non-increasing in $d$. The calculations below yield the uniform bound $B_{0.231}(x,y)\geq 1.0004$ over the whole admissible region. Since the finitely many cases vary continuously with $d$, a standard compactness argument shows that there exists $\eta$ such that $B_d(x,y)\geq 1 $ for every $d\leq 0.231+\eta$ and every admissible $(x,y)$, as desired. We now fix $d=0.231$.\bigskip

\textbf{Case 1.} If $x\leq d$ we have,
\[B_d(x,y)=\frac{f(x)}{x}\cdot (1-2d)\geq \frac{f(d)}{d}\cdot (1-2d) = \frac{1-2d}{2d}=1.1645>1,\]
where we used $d<\frac25$, hence $f(d)=\frac12$.\bigskip

\textbf{Case 2.} If $y\leq d\leq x$ we have,
\[B_d(x,y)=\brackets*{m(y)-\frac{f(x)}{x}}\cdot (x-d)
    +\frac{f(x)}{x}\cdot (1-2d),\]
    which is decreasing in $y$, hence
    \begin{align*}
        B_d(x,y)\geq h_d(x)
        &:= \brackets*{m(d)-\frac{f(x)}{x}}\cdot (x-d)
            +\frac{f(x)}{x}\cdot (1-2d)\\
        &=m(d)\cdot(x-d)+\frac{f(x)}{x}\cdot (1-x-d),        
    \end{align*}
where $3.414<m(d)<3.415$. We then have
\[h'_d(x) = m(d)+(1-d)\cdot \frac{f'(x)x-f(x)}{x^2}-f'(x).\]
Observe that $f'(x)=0$ if $x<\frac25$ or $x>\frac{7}{10}$, and $f'(x)=\frac{-10}{6}$ if $x\in(\frac25,\frac{7}{10})$. For $x\in[d,\frac25)$,
\[h'_d(x) = m(d)-\frac{1-d}{2x^2},\]
hence $h_d$ admits a unique minimum on that interval at $x_0=\sqrt{\frac{1-d}{2m(d)}}\approx0.3356$, satisfying
\[h_d(x_0)=m(d)(x_0-d)+\frac{1-x_0-d}{2x_0}\geq 1.0028 > 1.\]
For $x\in(\frac25,\frac{7}{10})$,
\[h'_d(x) = m(d)+\frac53-\frac{7(1-d)}{6x^2},\]
with a minimum on that range attained at $x_1=\sqrt{\frac{7(1-d)}{6m(d)+10}}\approx 0.4202$ satisfying
\[h_d(x_1)=m(d)(x_1-d)+\frac{7-10x_1}{6x_1}\cdot(1-x_1-d)\geq 1.0331 > 1.\]
Finally for $x\geq 7/10$, $f(x)=0$, hence 
\[h_d(x)= m(d)(x-d)\geq m(d)\parens*{\frac{7}{10}-d}\geq 1.6>1.\]

\textbf{Case 3.} If $d\leq y\leq x$ we have,
\begin{align*}
B_d(x,y)&=\brackets*{m(y)-\frac{f(x)}{x}}\cdot (x-d)
    +\brackets*{m(x)-\frac{f(x)}{x}}\cdot (y-d)
    +\frac{f(x)}{x}\cdot (1-2d)\\
    &= (x-d)\cdot m(y)+(y-d)\cdot m(x)+\frac{f(x)}{x}(1-x-y),
\end{align*}
with $2d\leq x+y\leq 1$. As $\frac{8}{35}\leq d\leq y\leq\min\{x,1-x\}$, we have $\frac{8}{35}\leq y\leq 1/2$ and therefore, 
\[m(y)=\frac{\sqrt{70}}{3\sqrt{y}}-\frac53-\frac{1}{6y}.\]
For fixed $x$ and $y\in(\frac{8}{35},\frac12]$ we have
\[\frac{\partial B_d}{\partial y}=(x-d)m'(y)+m(x)-\frac{f(x)}{x},
\qquad
m'(y)=\frac{1-\sqrt{70y}}{6y^2}<0\]
and
\[\frac{\partial^2 B_d}{\partial y^2}=(x-d)m{''}(y),
\qquad
m''(y)=\frac{3\sqrt{70y}-4}{12y^3}>0.\]
Therefore $B_d(x,\cdot)$ is convex on $(d,\min\{x,1-x\}]$, and its minimum is attained either at a boundary ($y=d$ or $y=\min\{x,1-x\}$) or at a unique interior point $y_0=y_0(x)$ with
\[(x-d)m'(y_0)+m(x)-\frac{f(x)}{x}=0.\]
Below we distinguish four cases depending on $x$. For each case, we substitute the explicit formulas for $m$ and $f$ and determine the minimum value of $B_d(x, y)$ by checking the three possibilities above. The numerical values in the four cases below were independently verified using both Python (SymPy and NumPy) and Mathematica. The Python computation solves the relevant stationary-point equations and additionally performs a dense grid search over each admissible region, see~\cref{app:CodePython}. Independently, we used Mathematica's \texttt{NMinimize} to minimise $B_d$ directly over each of the four constrained regions, using high-precision arithmetic, see~\cref{app:CodeMathematica}. Both computations give the values reported below.
\begin{enumerate}
    \item For $x\in[d,2/5]$, we have $x\leq 1-x$, $f(x)=1/2$, and 
    \[\min_{d\leq y\leq x}B_d(x,y)=B_d(x_0,d)\geq 1.0028 >1,\quad x_0=\sqrt{\frac{1-d}{2m(d)}}.\]
    \item For $x\in[2/5,1/2]$, we have $x\leq 1-x$, $f(x)=(7-10x)/6$, and
    \[\min_{d\leq y\leq x}B_d(x,y)=B_d(x^*,y^*)\geq 1.0004>1,\quad x^*\approx 0.4466,~y^*=y_0(x^*)\approx0.3437.\]
    \item For $x\in[1/2,7/10]$, we have $1-x\leq x$, $f(x)=(7-10x)/6$, and
    \[\min_{d\leq y\leq 1-x}B_d(x,y)=B_d(1/2,y_0(1/2))\geq 1.0203 >1.\]
    \item For $x\in[7/10,1-d]$, we have $1-x\leq x$, $f(x)=0$, and
    \[\min_{d\leq y\leq 1-x}B_d(x,y)=B_d(7/10,3/10)\geq 1.4443 >1.\qedhere\]
\end{enumerate}
\end{proof}

\section{Numerical verification - Python}\label{app:CodePython}

The following Python code, using \texttt{SymPy} and \texttt{NumPy}, was used to verify the numerical calculations from~\cref{sec:AppendixOptim} in the proof of~\cref{prop:optim}.

\begin{lstlisting}[
    style=coloredpython,%python,
    caption={Numerical verification of the optimisation in~\cref{sec:AppendixOptim}.},
    label={lst:optimisation}
]
import sympy as sp
import numpy as np

# Numerical verification of Appendix A, Case 3(i)--(iv)

#We fix d
d = 0.231

# f-function from (8)
def f(x):
    x = np.asarray(x, dtype=float)
    return np.where(x <= 2/5,0.5,np.where(x <= 0.7, (7 - 10*x)/6, 0.0))

# m(a) from the appendix. # In Case 3 we always have x,y >= d > 8/35, so only the
# second and third branches are actually needed.
def m(a):
    a = np.asarray(a, dtype=float)
    return np.where(a <= 8/35,1/(2*a) + 5/4,
        np.where(a <= 0.7,np.sqrt(70)/(3*np.sqrt(a)) - 5/3 - 1/(6*a),1/a))

def B(x, y):
    """B_d(x,y), with d fixed to 0.231."""
    return ((x-d)*m(y) + (y-d)*m(x) + f(x)/x*(1-x-y))

# ------------------------------------------------------------
# Useful SymPy expressions
# ------------------------------------------------------------

X, Y = sp.symbols('X Y', positive=True)

m_mid = sp.sqrt(70)/(3*sp.sqrt(X)) - sp.Rational(5, 3) - 1/(6*X)
m_mid_y = sp.sqrt(70)/(3*sp.sqrt(Y)) - sp.Rational(5, 3) - 1/(6*Y)
mprime_y = sp.diff(m_mid_y, Y)

md = float(m(d))
x0 = np.sqrt((1-d)/(2*md))

print(f"d ={d:.3f}, m(d)={md:.4f}, x0={float(x0):.4f}")

# ==============================================================================
# (i)  d <= y <= x <= 2/5 ; We verify that min B_d(x,y) = B_d(x0,d) >= 1.0028.
# ==============================================================================

val_i = B(x0, d)

# Independent dense-grid check over the whole triangular region.
xs = np.linspace(d, 0.4, 1501)
best_i = (np.inf, None, None)

for x in xs:
    ys = np.linspace(d, x, 1001)
    vals = B(x, ys)
    j = np.argmin(vals)
    if vals[j] < best_i[0]:
        best_i = (float(vals[j]), float(x), float(ys[j]))

print()
print("--------------------- Case 3(i) --------------------")
print("Claimed boundary minimiser in appendix")
print(f"  x0={float(x0):.4f}, y=d={d:.3f}, with B(x0,d)={float(val_i):.4f}")
print("Dense-grid minimum:")
print(f"  B={float(best_i[0]):.5f} at x={float(best_i[1]):.5f} and y ={float(best_i[2]):.5f}")

# ============================================================
# (ii)  2/5 <= x <= 1/2,  d <= y <= x
# For fixed x, an interior minimiser y0(x) satisfies
#   (x-d)m'(y) + m(x) - f(x)/x = 0.
# We solve simultaneously
#   dB/dy = 0, dB/dx = 0
# in the interior, then compare with the boundary.
# ============================================================

dS = sp.Float(d, 30)

f_mid_x = (7 - 10*X)/6
m_x = sp.sqrt(70)/(3*sp.sqrt(X)) - sp.Rational(5,3) - 1/(6*X)
m_y = sp.sqrt(70)/(3*sp.sqrt(Y)) - sp.Rational(5,3) - 1/(6*Y)

B_mid = ((X-dS)*m_y + (Y-dS)*m_x + f_mid_x/X*(1-X-Y))

Bx = sp.diff(B_mid, X)
By = sp.diff(B_mid, Y)

sol = sp.nsolve(
    [Bx, By],
    [X, Y],
    [0.446, 0.344],
    tol=1e-28,
    maxsteps=100,
    prec=50
)

x_star = float(sol[0])
y_star = float(sol[1])
val_ii = float(B_mid.subs({X: sol[0], Y: sol[1]}).evalf(30))

# Dense-grid check
xs = np.linspace(0.4, 0.5, 1501)
best_ii = (np.inf, None, None)

for x in xs:
    ys = np.linspace(d, x, 1501)
    vals = B(x, ys)
    j = np.argmin(vals)
    if vals[j] < best_ii[0]:
        best_ii = (float(vals[j]), float(x), float(ys[j]))

print()
print("--------------------- Case 3(ii) --------------------")
print("Interior stationary point:")
print(f"  x*={float(x_star):.4f}, y*={float(y_star):.4f}, B(x*,y*) ={float(val_ii):.4f}")
print("  dB/dx =", float(Bx.subs({X: sol[0], Y: sol[1]})))
print("  dB/dy =", float(By.subs({X: sol[0], Y: sol[1]})))
print()
print("Dense-grid minimum:")
print(f"  B={float(best_ii[0]):.5f} at x={float(best_ii[1]):.5f} and y ={float(best_ii[2]):.5f}")

# ============================================================
# (iii)  1/2 <= x <= 7/10,  d <= y <= 1-x
# The appendix says the minimum is attained at
# x = 1/2 and y = y0(1/2), where y0 solves dB/dy = 0.
# ============================================================

x_half = 0.5

# Solve the fixed-x stationarity equation in y.
By_half = sp.N(By.subs(X, sp.Rational(1,2)), 40)
y_half = sp.nsolve(By_half, Y, 0.33, prec=50)
y_half_f = float(y_half)
val_iii = B(x_half, y_half_f)

# Dense-grid check over the whole region.
xs = np.linspace(0.5, 0.7, 1501)
best_iii = (np.inf, None, None)

for x in xs:
    ymax = 1-x
    if ymax < d:
        continue
    ys = np.linspace(d, ymax, 1501)
    vals = B(x, ys)
    j = np.argmin(vals)
    if vals[j] < best_iii[0]:
        best_iii = (float(vals[j]), float(x), float(ys[j]))

print()
print("--------------------- Case 3(iii) --------------------")
print("Claimed minimiser:")
print(f"  x=1/2, y=y0(1/2)={y_half_f:.5f}, with B(1/2,y0)={float(val_iii):.4f}.")
print()
print("Dense-grid minimum:")
print(f"  B={float(best_iii[0]):.5f} at x={float(best_iii[1]):.5f} and y ={float(best_iii[2]):.5f}")


# ============================================================
# (iv)  7/10 <= x <= 1-d,  d <= y <= 1-x
#
# Here f(x)=0 and m(x)=1/x, while y <= 0.3 < 0.7.
# The appendix says the minimum is B(0.7,0.3).
# ============================================================

val_iv = B(0.7, 0.3)

xs = np.linspace(0.7, 1-d, 1501)
best_iv = (np.inf, None, None)

for x in xs:
    ymax = 1-x
    if ymax < d:
        continue
    ys = np.linspace(d, ymax, 1501)
    vals = B(x, ys)
    j = np.argmin(vals)
    if vals[j] < best_iv[0]:
        best_iv = (float(vals[j]), float(x), float(ys[j]))

print()
print("--------------------- Case 3(iv) ---------------------")
print("Claimed minimiser:")
print(f"  x=0.7, y=0.3, with B(0.7,0.3)={float(val_iv):.4f}.")
print()
print("Dense-grid minimum:")
print(f"  B={float(best_iv[0]):.5f} at x={float(best_iv[1]):.5f} and y ={float(best_iv[2]):.5f}")


# ===================================
# Final summary
# ===================================

print("--------------------- SUMMARY ---------------------")
print("(i)   minimum claimed >= 1.0028; computed:", val_i)
print("(ii)  minimum claimed >= 1.0004; computed:", val_ii)
print("(iii) minimum claimed >= 1.0203; computed:", val_iii)
print("(iv)  minimum claimed >= 1.4443; computed:", val_iv)
\end{lstlisting}

\subsection{Output}

\begin{lstlisting}[
    style=output,
    caption={Output of the Python numerical verification.},
    label={lst:python-output}
]
d = 0.231, m(d) = 3.4144, x0 = 0.3356

----------------------------------------------------
--------------------- Case 3(i) --------------------
----------------------------------------------------
Claimed boundary minimiser in appendix:
  x0 = 0.3356, y = d = 0.231, with B(x0,d) = 1.0029

Dense-grid minimum:
  B = 1.00286 at x = 0.33555 and y = 0.23100


----------------------------------------------------
-------------------- Case 3(ii) --------------------
----------------------------------------------------
Interior stationary point:
  x* = 0.4466, y* = 0.3437, B(x*,y*) = 1.0005
  dB/dx = -8.2937525312426315e-16
  dB/dy = -4.100376657139772e-17

Dense-grid minimum:
  B = 1.00047 at x = 0.44653 and y = 0.34365


----------------------------------------------------
------------------- Case 3(iii) --------------------
----------------------------------------------------
Claimed minimiser:
  x = 1/2, y = y0(1/2) = 0.38289,
  with B(1/2,y0) = 1.0203

Dense-grid minimum:
  B = 1.02032 at x = 0.50000 and y = 0.38290


----------------------------------------------------
-------------------- Case 3(iv) --------------------
----------------------------------------------------
Claimed minimiser:
  x = 0.7, y = 0.3, with B(0.7,0.3) = 1.4444

Dense-grid minimum:
  B = 1.44438 at x = 0.70000 and y = 0.30000


----------------------------------------------------
--------------------- SUMMARY ----------------------
----------------------------------------------------
(i)   claimed >= 1.0028; computed = 1.0028576228751145
(ii)  claimed >= 1.0004; computed = 1.000471331469278
(iii) claimed >= 1.0203; computed = 1.020322217607493
(iv)  claimed >= 1.4443; computed = 1.4443803184984163
\end{lstlisting}

\section{Numerical verification - Mathematica}\label{app:CodeMathematica}
\begin{lstlisting}[
    style=wolfram,
    caption={Independent Mathematica verification of the numerical optimisation.},
    label={lst:mathematica-verification}
]
(*Independent Mathematica verification of Appendix A,Case 3*)
ClearAll["Global`*"];
d = 231/1000;

(*Definitions*)
f[x_] := 
  Piecewise[{{1/2, x <= 2/5}, {(7 - 10 x)/6, 2/5 < x <= 7/10}, {0, 
     x > 7/10}}];
m[x_] := 
  Piecewise[{{1/(2 x) + 5/4, 
     x <= 8/35}, {Sqrt[70]/(3 Sqrt[x]) - 5/3 - 1/(6 x), 
     8/35 < x <= 7/10}, {1/x, x > 7/10}}];
B[x_, y_] := ((x - d) m[y] + (y - d) m[x] + (f[x]/x) (1 - x - y));

(*Direct constrained minimisation*)
case1 = NMinimize[{B[x, y], d <= y <= x <= 2/5}, {x, y}, 
   WorkingPrecision -> 30, AccuracyGoal -> 12, PrecisionGoal -> 12];
case2 = NMinimize[{B[x, y], 2/5 <= x <= 1/2 && d <= y <= x}, {x, y}, 
   WorkingPrecision -> 30, AccuracyGoal -> 12, PrecisionGoal -> 12];
case3 = NMinimize[{B[x, y], 1/2 <= x <= 7/10 && d <= y <= 1 - x}, {x, 
    y}, WorkingPrecision -> 30, AccuracyGoal -> 12, 
   PrecisionGoal -> 12];
case4 = NMinimize[{B[x, y], 
    7/10 <= x <= 1 - d && d <= y <= 1 - x}, {x, y}, 
   WorkingPrecision -> 30, AccuracyGoal -> 12, PrecisionGoal -> 12];

(*Results*)
Print["Case 3(i):   ", N[case1, 12]];
Print["Case 3(ii):  ", N[case2, 12]];
Print["Case 3(iii): ", N[case3, 12]];
Print["Case 3(iv):  ", N[case4, 12]];

(*Check the lower bounds stated in appendix*)

values = First /@ {case1, case2, case3, case4};

Print[""];
Print["Stated lower bounds:"];
Print["Case (i):   ", N[values[[1]], 12], " > 1.0028 : ", 
  values[[1]] > 1.0028];
Print["Case (ii):  ", N[values[[2]], 12], " > 1.0004 : ", 
  values[[2]] > 1.0004];
Print["Case (iii): ", N[values[[3]], 12], " > 1.0203 : ", 
  values[[3]] > 1.0203];
Print["Case (iv):  ", N[values[[4]], 12], " > 1.4443 : ", 
  values[[4]] > 1.4443];
\end{lstlisting}

\subsection{Ouput}
\begin{lstlisting}[
    style=output,
    caption={Output of the Mathematica numerical verification.},
    label={lst:Mathematica-output}
]

Case 3(i):   {1.00285762288,{x->0.335575022906,y->0.231000000000}}
Case 3(ii):  {1.00047133147,{x->0.446557148304,y->0.343704522940}}
Case 3(iii): {1.02032221761,{x->0.500000000000,y->0.382893115617}}
Case 3(iv):  {1.44438031850,{x->0.700000000000,y->0.300000000000}}

Stated lower bounds:
Case (i):   1.00285762288 > 1.0028 : True
Case (ii):  1.00047133147 > 1.0004 : True
Case (iii): 1.02032221761 > 1.0203 : True
Case (iv):  1.44438031850 > 1.4443 : True
\end{lstlisting}

\end{document}